\documentclass[12pt]{amsart}

\usepackage{graphicx,amsmath,amscd,amsfonts,amsthm,amssymb,verbatim,stmaryrd,fullpage,enumerate,tikz}
\newtheorem{theorem}{Theorem}[section]
\newtheorem{lemma}[theorem]{Lemma}
\newtheorem{prop}[theorem]{Proposition}

\newtheorem{cor}[theorem]{Corollary}

\theoremstyle{remark}
\newtheorem*{remark}{Remark}

\newcommand{\R}{\mathbb R}

\newcommand{\Z}{\mathbb Z}
\newcommand{\Q}{\mathbb Q}

\newcommand{\bH}{\mathbb H}

\newcommand{\g}{\mathfrak g}

\newcommand{\gk}{\mathfrak k}
\newcommand{\ga}{\mathfrak a}

\newcommand{\gn}{\mathfrak n}

\newcommand{\cD}{\mathcal D}
\newcommand{\cS}{\mathcal S}
\newcommand{\cT}{\mathcal T}

\newcommand{\cP}{\mathcal P}

\newcommand{\Ad}{\text{Ad}}

\newcommand{\tr}{\text{tr}}

\newcommand{\oa}{\overline{\alpha}}

\newcommand{\be}{\begin{equation}}
\newcommand{\ee}{\end{equation}}
\newcommand{\bes}{\begin{equation*}}
\newcommand{\ees}{\end{equation*}}

\newcommand{\ba}{\begin{eqnarray}}
\newcommand{\ea}{\end{eqnarray}}
\newcommand{\bas}{\begin{eqnarray*}}
\newcommand{\eas}{\end{eqnarray*}}

\title{Geodesic Restrictions of Arithmetic Eigenfunctions}

\author{Simon Marshall}
\address{Department of Mathematics\\
Northwestern University\\
2033 Sheridan Road\\
Evanston\\
IL 60208, USA}
\email{slm@math.northwestern.edu}
\thanks{Supported by NSF grant DMS-1201321.}

\begin{document}

\begin{abstract}
Let $X$ be an arithmetic hyperbolic surface, $\psi$ a Hecke-Maass form, and $\ell$ a geodesic segment on $X$.  We obtain a power saving over the local bound of Burq-G\'erard-Tzvetkov for the $L^2$ norm of $\psi$ restricted to $\ell$, by extending the technique of arithmetic amplification developed by Iwaniec and Sarnak.  We also improve the local bounds for various Fourier coefficients of $\psi$ along $\ell$.
\end{abstract}

\maketitle

\section{Introduction}

If $X$ is a compact Riemannian manifold and $\psi$ is a Laplace eigenfunction on $X$ satisfying $\Delta \psi = \lambda^2 \psi$, it is an interesting problem to study the extent to which $\psi$ can concentrate on small subsets of $X$.  Two well studied formulations of this problem are to normalise $\psi$ by $\| \psi \|_2 = 1$, and either bound $\| \psi \|_p$ for $2 \le p \le \infty$ or bound the $L^p$ norms of $\psi$ restricted to some submanifold.  We shall be interested in both of these problems in the case where $X$ is two dimensional and the submanifold we restrict to is a geodesic segment $\ell$.  The basic upper bound for $\| \psi \|_p$  in this case was proven by Sogge \cite{So} (see also Avakumovi\'c \cite{A} and Levitan \cite{L} when $p = \infty$), and is

\be
\label{sogge}
\| \psi \|_p \ll \lambda^{\delta(p)}
\ee
where $\delta(p)$ is given by

\bes
\delta(p) = \bigg\{ \begin{array}{ll} 1/2 - 2/p & p \ge 6
\\ 1/4 - 1/2p & 2 \le p \le 6. \end{array}
\ees
The standard bound for $\| \psi|_\ell \|_p$ is due to Burq, G\'erard and Tzvetkov \cite{BGT}, and is

\be
\label{BGT}
\| \psi|_\ell \|_p \ll \lambda^{\delta'(p)}
\ee
where $\delta'(p)$ is given by

\bes
\delta'(p) = \bigg\{ \begin{array}{ll} 1/2 - 1/p & p \ge 4
\\ 1/4 & 2 \le p \le 4. \end{array}
\ees
Both of these bounds are sharp when $X$ is the round 2-sphere, but can be strengthened under extra geometric assumptions on $X$ such as negative curvature, see for instance \cite{STZ, SZ1, SZ2}.  It should be noted that all such improvements in the negatively curved case are by at most a power of $\log \lambda$.\\

We now let $X$ be a compact arithmetic hyperbolic surface and $\psi$ a Hecke-Maass cusp form on $X$, which we shall always assume to be $L^2$-normalised.  In this case, Iwaniec and Sarnak \cite{IS} have shown that the bound $\| \psi \|_\infty \ll \lambda^{1/2}$ given by (\ref{sogge}) may be strengthened by a power to $\| \psi \|_\infty \ll_\epsilon \lambda^{5/12 + \epsilon}$.  Their approach, known as arithmetic amplification, is to construct a projection operator onto $\psi$ using the Hecke operators as well as the wave group.  It has been adapted by other authors to study the pointwise norms of arithmetic eigenfunctions in various aspects, see for instance \cite{BM, HT, Te} as well as the alternative approach taken in \cite{BH}.  In this paper we apply amplification to a new kind of semiclassical problem, namely improving the exponent in the bound (\ref{BGT}) for $\| \psi|_\ell \|_2$.  Our main result is as follows.

\begin{theorem}
\label{main}

Let $\psi$ be a Hecke-Maass eigenfunction on $X$ with spectral parameter $t$.  For any geodesic segment $\ell$ of unit length we have

\be
\label{subconv}
\| \psi|_\ell \|_2 \ll_\epsilon t^{3/14 + \epsilon},
\ee
where the implied constant is independent of $\ell$.

\end{theorem}

We may combine Theorem \ref{main} with a theorem of Bourgain \cite{Bo} to give an improvement over the local bound $\| \psi \|_4 \ll t^{1/8}$.

\begin{cor}
\label{L4}

We have $\| \psi \|_4 \ll_\epsilon t^{1/8 - 1/112 + \epsilon}$.

\end{cor}

Corollary \ref{L4} is much weaker than the bound $\| \psi \|_4 \ll_\epsilon t^\epsilon$ announced by Sarnak and Watson (\cite{Sa}, Theorem 3), although their result may be conditional on the Ramanujan conjecture.  See also \cite{BKY} for results in the case of holomorphic eigenforms.  Note that Bourgain's theorem actually gives an equivalence (up to factors of $t^\epsilon$) between a sub-local bound for $\| \psi \|_4$ and one for $\| \psi |_\ell \|_2$ that is uniform in $\ell$, and so the bound of Sarnak and Watson implies Theorem \ref{main} with an exponent of $1/8$.  However, we feel that our method is of interest as it does not rely on special value identities or summation formulas, and we hope to apply it to restriction problems on other groups by combining it with the techniques of \cite{Ma}.\\

The methods we use to prove Theorem \ref{main} also allow us to prove bounds for periods of $\psi$ along $\ell$.  We let $\ell : [0,1] \rightarrow X$ be an arc length parametrisation of $\ell$, and let $b \in C^\infty_0(\R)$ be a function with $\text{supp}(b) \subset [0,1]$.  For $1/2 > \delta > 0$, let $I_\delta = [-1+\delta, -\delta] \cup [\delta, 1-\delta]$.

\begin{theorem}
\label{period}

For $\lambda \in \R$, denote the integral

\bes
\int_{-\infty}^\infty e^{i\lambda x} b(x) \psi(\ell(x)) dx
\ees
by $\langle \psi, b e^{i\lambda x} \rangle$.

\begin{enumerate}[(a)]

\item If $\lambda = 0$ we have $\langle \psi, b e^{i\lambda x} \rangle \ll_\epsilon t^{-1/12 + \epsilon}$.

\item If $1/2 > \delta > 0$ and $\lambda / t \in I_\delta$, we have $\langle \psi, b e^{i\lambda x} \rangle \ll_\epsilon t^{-1/18 + \epsilon}$.

\item Define $\beta = \min |\lambda \pm t|$.  If $\beta \le t^{2/3}$, we have $\langle \psi, b e^{i\lambda x} \rangle \ll_\epsilon t^{5/24 + \epsilon} (1 + \beta)^{1/24}$.

\end{enumerate}
All of these bounds are uniform in $\lambda$ and $\ell$.

\end{theorem}

\begin{remark}

The bound $\beta \le t^{2/3}$ in Theorem \ref{period} could be replaced with $t^{1-\delta}$ for any $\delta > 0$, however when $\beta \ge t^{1/7 + \epsilon}$ the bound $\langle \psi, b e^{i\lambda x} \rangle \ll_\epsilon t^{5/24 + \epsilon} (1 + \beta)^{1/24}$ is weaker than the local bound of Proposition \ref{L2offspec}.

\end{remark}

When $\ell$ is a closed geodesic instead of a segment, cases (a) and (b) of Theorem \ref{period} may be compared with the local bound $\langle \psi, b e^{i\lambda x} \rangle  \ll 1$ given in \cite{Re,Z}, and the improvement $\langle \psi, b \rangle = o(1)$ given in \cite{CS} in the case of negative curvature.  These cases should correspond via a formula of Waldspurger \cite{W} to a subconvex bound for certain $L$-values of the form $L(1/2, \psi \otimes \theta_\chi)$, where $\chi$ is a Grossencharacter of a real quadratic field and $\theta_\chi$ is the associated theta series on $GL_2$.\\

As in \cite{IS}, Theorems \ref{main} and \ref{period} can both be strengthened under the assumption that the Fourier coefficients of $\psi$ are not small.  In the case of Theorem \ref{main} and case (c) of Theorem \ref{period}, this assumption allows us to employ an amplifier of sufficient length that it becomes profitable to estimate the Hecke recurrence using spectral methods, rather than the standard diophantine ones.  Let $\lambda(n)$ be the automorphically normalised Hecke eigenvalues of $\psi$, and assume that they satisfy the bounds

\be
\label{thick}
\sum_{N \le p \le 2N} | \lambda(p)| \gg_\epsilon N^{1-\epsilon}
\ee
for all $N \ge 2$ and

\be
\label{raman}
|\lambda(p)| \le 2 p^\theta
\ee
for some $\theta < 1/2$ and $p$ prime.  Note that (\ref{raman}) is known with $\theta = 7/64$, see \cite{BB}.  We then prove

\begin{theorem}
\label{conditional}

If the normalised Hecke eigenvalues $\lambda(n)$ satisfy (\ref{thick}) and (\ref{raman}) we have

\bes
\| \psi|_\ell \|_2 \ll_\epsilon t^{1/(8 - 8\theta) + \epsilon},
\ees
while if $\beta = \min |\lambda \pm t|$ and $\beta \le t^{2/3}$ we have

\be
\label{conditionalperiod}
\langle \psi, b e^{i\lambda x} \rangle \ll_\epsilon t^{\theta/2 + \epsilon} (1 + \beta)^{1/4 - \theta/2},
\ee
uniformly in $\lambda$ and $\ell$.

\end{theorem}

In particular, Theorem \ref{conditional} gives $\langle \psi, b e^{i\lambda x} \rangle \ll_\epsilon t^\epsilon$ when $|\lambda - t| \ll t^\epsilon$ under the assumption that $\theta = 0$.  We note that (\ref{conditionalperiod}) becomes weaker than the local bound of Proposition \ref{L2offspec} when $\beta \ge t^{1/2 + \epsilon}$.

\textit{Acknowledgements.}  We would like to thank Xiaoqing Li, Peter Sarnak, Christopher Sogge, Nicolas Templier, Akshay Venkatesh, and Steve Zelditch for many helpful discussions.

\section{Notation}
\label{notation}

For simplicity, we shall restrict attention to $X$ that arise from a quaternion division algebra $A = ( \frac{a,b}{\Q} )$ over $\Q$.  Here $a, b \in \Z$ are square free and we will assume that $a > 0$.  We choose a basis $1, \omega, \Omega, \omega \Omega$ for $A$ over $\Q$ that satisfies $\omega^2 = a$, $\Omega^2 = b$ and $\omega \Omega + \Omega \omega = 0$.  We denote the norm and trace by $N(\alpha) = \alpha \oa$ and $\tr(\alpha) = \alpha + \oa$.  We let $R$ be a maximal order in $A$ (or more generally an Eichler order, see \cite{E}), and for $m \ge 1$ let

\bes
R(m) = \{ \alpha \in R | N(\alpha) = m \}.
\ees
$R(1)$ is the group of elements of norm 1; it acts on $R(m)$ by multiplication on the left and $R(1) \backslash R(m)$ is known to be finite \cite{E}.  Fix an embedding $\phi : A \rightarrow M_2(F)$, the $2 \times 2$ matrices with entries in $F = \Q( \sqrt{a})$ by

\bes
\phi(\alpha) = \left( \begin{array}{cc} \overline{\xi} & \eta \\ b \overline{\eta} & \xi \end{array} \right)
\ees
where

\bes
\alpha = x_0 + x_1 \omega + (x_2 + x_3 \omega)\Omega = \xi + \eta \Omega.
\ees
We define the lattice $\Gamma = \phi(R(1)) \subset SL(2,\R)$, which is co-compact as we assumed $A$ to be a division algebra, and let $X = \Gamma \backslash \bH$.  We define the Hecke operators $T_n: L^2(X) \rightarrow L^2(X)$, $n \ge 1$, by

\bes
T_nf (z) = \sum_{\alpha \in R(1) \backslash R(n)} f( \phi(\alpha) z).
\ees
There is a positive integer $q$ (depending on $R$) such that for $(n,q) = 1$, $T_n$ has the following properties (see \cite{E}):

\begin{align*}
T_n = T_n^*, & \quad \text{that is $T_n$ is self-adjoint}, \\
T_n T_m & = \sum_{d | (n,m)} d T_{nm/d^2}.
\end{align*}
We let $\lambda(n)$ be the normalised Hecke eigenvalues of $\psi$ and $t$ be its spectral parameter, so that

\begin{align*}
T_n \psi & = \lambda(n) n^{1/2} \psi, \\
\Delta \psi & = (1/4 + t^2) \psi.
\end{align*}

We let $K$, $A$, and $N$ be the standard subgroups of $PSL_2(\R)$, with parametrisations

\bes
k(\theta) = \left( \begin{array}{cc} \cos \theta/2 & \sin \theta/2 \\ -\sin \theta/2 & \cos \theta/2 \end{array} \right), \qquad a(y) = \left( \begin{array}{cc} e^y & 0 \\ 0 & 1 \end{array} \right), \qquad n(x) = \left( \begin{array}{cc} 1 & x \\ 0 & 1 \end{array} \right).
\ees
In particular, $k(\theta)$ represents an anticlockwise rotation by $\theta$ about the point $i$.  We denote the Lie algebra of $PSL_2(\R)$ by $\g$, and equip $\g$ with the norm

\be
\label{gnorm}
\| \cdot \| : \left( \begin{array}{cc} X_1 & X_2 \\ X_3 & -X_1 \end{array} \right) \mapsto \sqrt{X_1^2 + X_2^2 + X_3^2}.
\ee
This norm defines a left-invariant metric on $PSL_2(\R)$, which we denote by $d_G$.  We denote the Lie algebras of $K$, $A$, and $N$ by $\gk$, $\ga$, and $\gn$, and write the Iwasawa decomposition as

\be
\label{Iwasawa}
g = n(g) \exp( A(g)) k(g) = \exp( N(g)) \exp( A(g)) k(g).
\ee
We define

\be
\label{Hdef}
H = \left( \begin{array}{cc} 1/2 & 0 \\ 0 & -1/2 \end{array} \right) \in \ga, \quad X_\gn = \left( \begin{array}{cc} 0 & 1 \\ 0 & 0 \end{array} \right) \in \gn, \quad X_\gk = \left( \begin{array}{cc} 0 & 1/2 \\ -1/2 & 0 \end{array} \right).
\ee
We identify $\ga \simeq \R$ under the map $H \mapsto 1$, and consider $A(g)$ as a function $A : PSL_2(\R) \rightarrow \R$ under this identification, and likewise for $\gn$ and $N(g)$.  We let $\varphi_s$ denote the standard spherical function with spectral parameter $s$ on $\bH$ or $PSL_2(\R)$, depending on the context.

Throughout the paper, the notation $A \ll B$ will mean that there is a positive constant $C$ such that $|A| \le CB$, and $A \sim B$ will mean that there are positive constants $C_1$ and $C_2$ such that $C_1 B \le A \le C_2 B$.

\section{Amplification of geodesic periods}
\label{ampperiod}

We now prove cases (a) and (b) of Theorem \ref{period}.  As we may assume that $\psi$ is real, we may also assume that $\lambda \ge 0$.  We shall fix $1/2 > \delta > 0$, and assume that either $\lambda / t \in [\delta, 1-\delta]$ or $\lambda = 0$.

Let $h \in \cS(\R)$ be a real-valued function of Payley-Wiener type that is positive, even, and $\ge 1$ in the interval $[-1, 1]$.  Define $h_t$ by $h_t(s) = h(s - t) + h( -s - t)$, and let $k^0_t$ be the $K$-biinvariant function on $\bH$ with Harish-Chandra transform $h_t$ (see \cite{GV} or \cite{Se} for definitions).  The Payley-Wiener theorem of Gangolli \cite{Ga} implies that $k^0_t$ is of compact support that may be chosen arbitrarily small.  Let $K^0_t$ be the point-pair invariant on $\bH$ associated to $k^0_t$, which is real-valued and satisfies $K_t^0(x,y) = K_t^0(y,x)$.  Let $A^0_t$ the operator on $X$ with integral kernel

\bes
A_t^0(x,y) = \sum_{\gamma \in \Gamma} K_t^0(x, \gamma y).
\ees
It follows that $A_t^0$ is a self-adjoint approximate spectral projector onto the eigenfunctions in $L^2(X)$ with spectral parameter near $t$.  Let $k_t$ be the $K$-biinvariant function on $\bH$ with Harisch-Chandra transform $h_t^2$, and let $K_t$ and $A_t$ be associated to $k_t$ in the same way.  It follows that $A_t = (A_t^0)^2$.

Let $\ell \subset \bH$ be a unit length geodesic segment.  By abuse of notation, we also let $\ell : [0,1] \rightarrow \bH$ be an arc length parametrisation of $\ell$.  Let $b \in C^\infty_0(\R)$ be a function with $\text{supp}(b) \subset [0,1]$, and let $\lambda \in \R$.  Let $N \ge 1$ be an integer, and let $\alpha_n$, $n \le N$, be a sequence of complex numbers.  We define $\cT$ to be the Hecke operator

\bes
\cT = \sum_{1 \le n \le N} \frac{\alpha_n}{\sqrt{n}} T_n.
\ees

We shall estimate $\langle \psi, b e^{i\lambda x} \rangle$ by estimating $\langle \cT A_t^0 \psi, b e^{i\lambda x} \rangle$.  We first take adjoints to obtain $\langle \cT A_t^0 \psi, b e^{i\lambda x} \rangle = \langle \psi, \cT^* A_t^0 b e^{i\lambda x} \rangle$, where $A_t^0 b e^{i\lambda x}$ is the function on $X$ given by

\bes
A_t^0 b e^{i\lambda x}(y) = \int_{-\infty}^\infty A_t^0(y,\ell(x)) b(x) e^{i\lambda x} dx.
\ees
We then apply Cauchy-Schwarz to obtain

\begin{align*}
|\langle \psi, \cT^* A_t^0 b e^{i\lambda x} \rangle| & \le \langle \cT^* A_t^0 b e^{i\lambda x}, \cT^* A_t^0 b e^{i\lambda x} \rangle^{1/2} \\
& = \langle b e^{i\lambda x}, \cT \cT^* A_t b e^{i\lambda x} \rangle.
\end{align*}
We have

\bes
\cT \cT^* = \sum_{m,n \le N} \alpha_n \oa_m \sum_{ d | (n,m) } \frac{d}{\sqrt{mn}} T_{nm/d^2},
\ees
and so

\bes
\cT \cT^* A_t b e^{i\lambda x}(y) = \sum_{m,n \le N} \alpha_n \oa_m \sum_{ d | (n,m) } \frac{d}{\sqrt{mn}} \sum_{\gamma \in R(nm/d^2)} \int_{-\infty}^\infty K_t(y, \gamma \ell(x)) b(x) e^{i\lambda x} dx.
\ees
If $\ell_1$ and $\ell_2$ are a pair of unit geodesic segments in $\bH$ with parametrisations $\ell_i : [0,1] \rightarrow \bH$, we define

\bes
I(t, \lambda, \ell_1, \ell_2) = \iint_{-\infty}^\infty b(x_1) b(x_2) e^{i\lambda(x_1-x_2)} K_t(\ell_1(x_1), \ell_2(x_2)) dx_1 dx_2.
\ees
With this notation, we have

\be
\label{Heckesum1}
\langle b e^{i\lambda x}, \cT \cT^* A_t b e^{i\lambda x} \rangle = \sum_{m,n \le N} \alpha_n \oa_m \sum_{ d | (n,m) } \frac{d}{\sqrt{mn}} \sum_{\gamma \in R(nm/d^2)} I(t, \lambda, \ell, \gamma \ell).
\ee

To estimate the integrals $I(t, \lambda, \ell, \gamma \ell)$, we introduce two distance functions on pairs of unit geodesics.  Let $\ell_0$ be the upwards pointing unit geodesic based at $i$, and let $\ell_1 = g_1 \ell_0$ and $\ell_2 = g_2 \ell_0$.  We define

\bes
d(\ell_1, \ell_2) = \text{inf} \{ d(p,q) | p \in \ell_1, q \in \ell_2 \},
\ees
where $d(p,q)$ is the hyperbolic distance between points.  We also define

\bes
n(\ell_1, \ell_2) = \text{inf} \{ d_G( g_1^{-1} g_2, a) | a \in A \}.
\ees
In particular, $n(\ell_1, \ell_2) = 0$ iff the infinite extensions of $\ell_1$ and $\ell_2$ coincide and have the same orientation.  We assume that $k_t$ is supported in a ball of radius 1 about $i$, so that $I(t, \lambda, \ell_1, \ell_2) = 0$ unless $d(\ell_1, \ell_2) \le 1$.  We shall prove the following bounds for $I(t, \lambda, \ell_1, \ell_2)$.

\begin{prop}
\label{Ibound1}

Suppose $d(\ell_1, \ell_2) \le 1$.  If $\lambda / t \in [\delta, 1-\delta]$, we have

\bes
I(t, \lambda, \ell_1, \ell_2) \ll \Bigg\{ \begin{array}{cc} (1 + t n(\ell_1, \ell_2) )^{-1/2} & n(\ell_1, \ell_2) \le t^{-1/3} \\
t^{-1/3} & n(\ell_1, \ell_2) \ge t^{-1/3}. \end{array}
\ees
If $\lambda = 0$, we have

\bes
I(t, \lambda, \ell_1, \ell_2) \ll (1 + t n(\ell_1, \ell_2) )^{-1/2}.
\ees

\end{prop}

The second result we shall need is a bound for the counting function

\bes
M( \ell, n, \kappa) = | \{ \gamma \in R(n) | d( \gamma \ell, \ell) \le 1, n(\ell, \gamma \ell) < \kappa \} |.
\ees

\begin{lemma}
\label{Hecke}

We have the bound

\bes
M( \ell, n, \kappa) \ll_\epsilon (\kappa^2 + \kappa^{1/2})n^{1+\epsilon} + n^\epsilon
\ees
uniformly in $\ell$.

\end{lemma}

\begin{proof}

This may be proven in exactly the same way as the corresponding Lemma 1.3 of \cite{IS}.  The only differences are that we must consider the quadratic form $[ \alpha, \beta, \gamma]$ associated to $\ell$ with 

\bes
\beta^2 - 4 \alpha \gamma = 1,
\ees
and the subgroup $K_\ell$ generated by translation along $\ell$ which may be parametrized as

\bes
K_\ell = \left\{ \left[ \begin{array}{cc} t - \beta u & -2\gamma u \\ 2 \alpha u & t + \beta u \end{array} \right] | t^2 - u^2 = 1 \right\}.
\ees
As $\Gamma$ was cocompact, we may assume that $\ell$ lies in a fixed compact set.  If $d( \ell, \gamma \ell) \le 1$, we have

\bes
n( \ell, \gamma \ell) < \kappa \rightarrow \gamma = z + O(\kappa)\quad \text{with} \quad z \in K_\ell.
\ees
If we write $\gamma$ as

\bes
\gamma = \frac{1}{\sqrt{n}} \left[ \begin{array}{cc} x_0 - x_1 \sqrt{a} & x_2 + x_3 \sqrt{a} \\
bx_2 - bx_3\sqrt{a} & x_0 + x_1 \sqrt{a} \end{array} \right]
\ees
then $x_0$ and $x_1$ must satisfy the equations

\bes
\big| x_0^2 - \frac{a}{\beta^2} x_1^2 - n \big| \ll n \kappa, \quad |x_0| \ll \sqrt{n}, \quad |x_1| \ll \sqrt{n},
\ees
where the last two conditions come from the fact that the entries of $\gamma$ must be bounded.  The proof now proceeds exactly as in \cite{IS}, with the difference that we must count ideals of a given norm in real quadratic fields rather than imaginary ones, and the presence of units intorduces an extra factor of $n^\epsilon$ into our counting which we may ignore.

\end{proof}

With these results, we are ready to estimate the sum (\ref{Heckesum1}).  We first consider the case in which $\lambda / t \in [\delta, 1-\delta]$.  If we assume that $d(\ell, \gamma \ell) \le 1$ then we have $n(\ell, \gamma \ell) \in [0, 2]$, and we cover $[0,2]$ with the intervals $I_0 = [0, t^{-1}]$, $I_k = [e^{k-1} t^{-1}, e^k t^{-1} ]$ for $1 \le k \le \tfrac{2}{3} \log t$, and $I_\infty = [e^{-1} t^{-1/3}, 2 ]$.  When $n(\ell, \gamma \ell) \in I_0$ we apply the bounds

\begin{align*}
|I(t, \lambda, \ell_1, \ell_2)| & \ll 1 \\
M( \ell, n, t^{-1}) & \ll t^{-1/2} n^{1+\epsilon} + n^\epsilon
\end{align*}
from Proposition \ref{Ibound1} and Lemma \ref{Hecke} to obtain

\begin{align*}
\notag
\sum_{m,n \le N} \alpha_n \oa_m \sum_{ d | (n,m) } \frac{d}{\sqrt{mn}} \sum_{ \substack{ \gamma \in R(nm/d^2),\\  n( \ell, \gamma \ell) \in I_0 } } I(t, \lambda, \ell, \gamma \ell) & \ll N^\epsilon t^\epsilon \sum_{m,n \le N} \alpha_n \oa_m \sum_{ d | (n,m) } \frac{d}{\sqrt{mn}} \left( t^{-1/2} \frac{nm}{d^2} + 1 \right) \\
\label{k0}
& \ll N^\epsilon t^\epsilon \sum_{m,n \le N} \alpha_n \oa_m \sum_{ d | (n,m) } \frac{\sqrt{mn}}{d} t^{-1/2} + \frac{d}{\sqrt{mn}}.
\end{align*}
When $n(\ell, \gamma \ell) \in I_k$ we have

\begin{align*}
|I(t, \lambda, \ell_1, \ell_2)| & \ll e^{-k/2} \\
M( \ell, n, e^k t^{-1}) & \ll t^{-1/2} e^{k/2} n^{1+\epsilon} + n^\epsilon,
\end{align*}
which gives

\begin{align*}
\notag
\sum_{m,n \le N} \alpha_n \oa_m \sum_{ d | (n,m) } \frac{d}{\sqrt{mn}} \sum_{ \substack{ \gamma \in R(nm/d^2),\\  n( \ell, \gamma \ell) \in I_k } } I(t, \lambda, \ell, \gamma \ell) & \ll N^\epsilon t^\epsilon \sum_{m,n \le N} \alpha_n \oa_m \sum_{ d | (n,m) } \frac{d}{\sqrt{mn}} \left( t^{-1/2} \frac{nm}{d^2} + e^{-k/2} \right) \\
& \ll N^\epsilon t^\epsilon \sum_{m,n \le N} \alpha_n \oa_m \sum_{ d | (n,m) } \frac{\sqrt{mn}}{d} t^{-1/2} + \frac{d}{\sqrt{mn}} e^{-k/2}.
\end{align*}
When $n(\ell, \gamma \ell) \in I_\infty$ we have

\begin{align*}
|I(t, \lambda, \ell_1, \ell_2)| & \ll t^{-1/3} \\
M( \ell, n, 10) & \ll n^{1+\epsilon},
\end{align*}
so that

\bes
\sum_{m,n \le N} \alpha_n \oa_m \sum_{ d | (n,m) } \frac{d}{\sqrt{mn}} \sum_{ \substack{ \gamma \in R(nm/d^2),\\  n( \ell, \gamma \ell) \in I_\infty } } I(t, \lambda, \ell, \gamma \ell) \ll N^\epsilon t^\epsilon \sum_{m,n \le N} \alpha_n \oa_m \sum_{ d | (n,m) } \frac{\sqrt{mn}}{d} t^{-1/3}.
\ees

Combining these, and noting that we are summing over $\ll \log t$ values of $k$, we obtain

\be
\label{Heckesum2}
\langle b e^{i\lambda x}, \cT \cT^* A_t b e^{i\lambda x} \rangle \ll N^\epsilon t^\epsilon \sum_{m,n \le N} \alpha_n \oa_m \sum_{ d | (n,m) } \frac{\sqrt{mn}}{d} t^{-1/3} + \frac{d}{\sqrt{mn}}.
\ee
As in \cite{IS}, p. 310, we have

\be
\label{alpha1}
\sum_{m, n \le N} \sum_{d | (n,m) } \frac{\sqrt{nm}}{d} | \alpha_n \alpha_m| \le N^{1+\epsilon} \left( \sum_{n \le N} |\alpha_n| \right)^2,
\ee
and

\be
\label{alpha2}
\sum_{m, n \le N} \sum_{d | (m,n)} \frac{d}{ \sqrt{mn}} |\alpha_n \alpha_m| \ll N^\epsilon \sum_{n \le N} |\alpha_n|^2.
\ee
Combining (\ref{Heckesum2}) with (\ref{alpha1}) and (\ref{alpha2}) gives

\bes
\langle b e^{i\lambda x}, \cT \cT^* A_t b e^{i\lambda x} \rangle \ll N^\epsilon t^\epsilon \left( \sum_{n \le N} |\alpha_n|^2 + N t^{-1/3} \left( \sum_{n \le N} |\alpha_n| \right)^2 \right).
\ees
If we choose $\{ \alpha_n \}$ to be the amplifier used in \cite{IS}, it follows as on p. 311 there that

\bes
|\langle \psi, b e^{i\lambda x} \rangle|^2 \ll N^\epsilon t^\epsilon ( N^{-1/2} + N t^{-1/3}),
\ees
and choosing $N = t^{2/9}$ completes the proof.

The proof in the case $\lambda = 0$ is almost identical.  We again perform a dyadic sum over $n(\ell, \gamma \ell)$ and simplify to obtain

\bes
\langle b, \cT \cT^* A_t b \rangle \ll N^\epsilon t^\epsilon \left( \sum_{n \le N} |\alpha_n|^2 + N t^{-1/2} \left( \sum_{n \le N} |\alpha_n| \right)^2 \right),
\ees
and the result follows by using the same amplifier with $N = t^{1/3}$.

\section{Bounds for $L^2$ norms}

To prove Theorem \ref{main}, it suffices to bound the $L^2$ norm of $b(x) \psi(\ell(x)) \in L^2(\R)$ for $b \in C^\infty_0(\R)$ with $\text{supp}(b) \subseteq [0,1]$, provided the bound is uniform in $\ell$.  If $f \in C^\infty_0(\R)$, define its Fourier transform $\widehat{f}$ by

\bes
\widehat{f}(\xi) = \int_{-\infty}^\infty f(x) e^{-i \xi x} dx,
\ees
and extend this to an operator on $L^2(\R)$.  Let $\beta$ be a parameter satisfying $1 \le \beta \le t^{2/3}$.  Define $H_\beta^+$, $H_\beta^- \subset L^2(\R)$ to be the spaces of functions whose Fourier support lies in $[ \pm t-\beta, \pm t+\beta]$, and define $H_\beta = H_\beta^+ + H_\beta^-$.  Let $\Pi_\beta$ be the orthogonal projection onto $H_\beta$, and likewise for $\Pi_\beta^\pm$ and $H_\beta^\pm$.  We shall bound $\Pi_\beta b\psi$ and $(1 - \Pi_\beta) b\psi$ separately, by applying amplification to the former and a local bound to the latter, and as $\psi$ is real-valued it suffices to bound $\Pi_\beta^+ b\psi$.  The results we are obtain are the following.

\begin{prop}
\label{L2onspec}

We have $\| \Pi_\beta^+ b \psi \|_2 \ll_\epsilon t^{5/24 + \epsilon} \beta^{1/24}$, uniformly in $\beta$ and $\ell$.

\end{prop}

\begin{prop}
\label{L2offspec}

We have $\| (1 - \Pi_\beta) b \psi \|_2 \ll_\epsilon t^{1/4 + \epsilon} \beta^{-1/4}$, uniformly in $\beta$ and $\ell$.

\end{prop}

Combining these two results with $\beta = t^{1/7}$ gives Theorem \ref{main}.  Note that we expect Proposition \ref{L2offspec} to be sharp on the round sphere.\\

\subsection{Amplification of geodesic periods with $\lambda \sim t$}
\label{sect3}

We shall prove Proposition \ref{L2onspec} using the method of Section \ref{ampperiod}.  As before, it suffices to estimate $\langle \psi, b \phi \rangle$ for $\phi \in H_\beta^+$ with $\| \phi \|_2 = 1$, and we have

\bes
|\langle \cT A_t^0 \psi, b \phi \rangle|^{1/2} \le \langle b \phi, \cT \cT^* A_t b \phi \rangle.
\ees
If $\ell_1$ and $\ell_2$ are a pair of unit geodesic segments in $\bH$ with parametrisations $\ell_i : [0,1] \rightarrow \bH$, we define

\bes
I(t, \phi, \ell_1, \ell_2) = \iint_{-\infty}^\infty b(x_1) b(x_2) \phi(x_1) \overline{\phi(x_2)} K_t(\ell_1(x_1), \ell_2(x_2)) dx_1 dx_2.
\ees
With this notation, we again have

\be
\label{equalsum1}
\langle b \phi, \cT \cT^* A_t b \phi \rangle = \sum_{m,n \le N} \alpha_n \oa_m \sum_{ d | (n,m) } \frac{d}{\sqrt{mn}} \sum_{\gamma \in R(nm/d^2)} I(t, \phi, \ell, \gamma \ell).
\ee

We let the geodesic distance functions $d(\ell_1, \ell_2)$ and $n(\ell_1, \ell_2)$ be as in Section \ref{ampperiod}.  The estimate for $I(t, \phi, \ell_1, \ell_2)$ corresponding to Proposition \ref{Ibound1} in this case is as follows.

\begin{prop}
\label{Ibound2}

Suppse $d(\ell_1, \ell_2) \le 1$.  We have

\be
\label{Ibound21}
| I(t, \phi, \ell_1, \ell_2) | \ll t^{1/2}
\ee
for all $\ell_1$ and $\ell_2$, while if $n(\ell_1, \ell_2) \ge t^{-1/2 + \epsilon} \beta^{1/2}$ we have

\be
\label{Ibound22}
| I(t, \phi, \ell_1, \ell_2) | \ll_{\epsilon, A} t^{-A}.
\ee
The implied constants in both bounds are independent of $\phi$ and $\beta$.

\end{prop}

We shall prove Proposition \ref{Ibound2} in Section \ref{sect4}.  Proposition \ref{Ibound2} implies that we only need to consider the terms in (\ref{equalsum1}) with $d(\ell, \gamma \ell) \le 1$ and $n(\ell, \gamma \ell) \le t^{-1/2 + \epsilon} \beta^{1/2}$.  Lemma \ref{Hecke} gives

\bes
M(\ell, n, t^{-1/2+\epsilon} \beta^{1/2} ) \ll_\epsilon t^{-1/4 + \epsilon} \beta^{1/4} n^{1+\epsilon} + n^\epsilon,
\ees
and so we have

\begin{align}
\notag
\sum_{m,n \le N} \alpha_n \oa_m \sum_{ d | (n,m) } \frac{d}{\sqrt{mn}} \sum_{ \gamma \in R(nm/d^2) } I(t, \phi, \ell, \gamma \ell) & \ll N^\epsilon t^\epsilon \sum_{m,n \le N} \alpha_n \oa_m \sum_{ d | (n,m) } \frac{d}{\sqrt{mn}} t^{1/2} M(\ell, n, t^{-1/2+\epsilon} \beta^{1/2} ) \\
\label{k0}
& \ll N^\epsilon t^\epsilon \sum_{m,n \le N} \alpha_n \oa_m \sum_{ d | (n,m) } \frac{\sqrt{mn}}{d} t^{1/4} \beta^{1/4} + \frac{d}{\sqrt{mn}}t^{1/2}.
\end{align}
Combining (\ref{k0}) with (\ref{alpha1}) and (\ref{alpha2}) gives

\bes
\langle b \phi, \cT \cT^* A_t b \phi \rangle \ll N^\epsilon t^\epsilon \left( t^{1/2} \sum_{n \le N} |\alpha_n|^2 + N t^{1/4} \beta^{1/4} \left( \sum_{n \le N} |\alpha_n| \right)^2 \right),
\ees
and Proposition \ref{L2onspec} now follows as in Section \ref{ampperiod} by choosing $N = t^{1/6} \beta^{-1/6}$.

\subsection{Bounds Away from the Spectrum}
\label{boundsoffspec}

We now give the proof of Proposition \ref{L2offspec}.  We are free to assume that $\beta \ge 2t^\epsilon$, as otherwise the result follows from the bound (\ref{BGT}) of Burq-G\'erard-Tzvetkov.  As we will not be using Hecke operators, we are free to replace $\Gamma$ by a finite index sublattice with $\text{inj rad}(X) \ge 10$.  It suffices to estimate $\langle \psi, b \phi \rangle$ for $\phi \in H_\beta^\perp$ with $\| \phi \|_2 = 1$.  Let $k_t$, $K_t$ and $A_t$ be as in Section \ref{ampperiod}.  It follows as before that

\bes
|\langle A_t^0 \psi, b \phi \rangle| \le \langle b \phi, A_t b \phi \rangle^{1/2},
\ees
where

\bes
\langle b \phi, A_t b \phi \rangle = \iint_{-\infty}^\infty b(x_1) b(x_2) \phi(x_1) \overline{\phi(x_2)} \sum_{\gamma \in \Gamma} K_t(\ell(x_1), \gamma \ell(x_2)) dx_1 dx_2.
\ees
Our assumptions that $\text{inj rad}(X) \ge 10$ and $k_t$ is supported in a ball of radius 1 imply that only the term $\gamma = e$ makes a contribution to the inner sum, so that

\bes
\langle b \phi, A_t b \phi \rangle = \iint_{-\infty}^\infty b(x_1) b(x_2) \phi(x_1) \overline{\phi(x_2)} K_t(\ell(x_1), \ell(x_2)) dx_1 dx_2.
\ees
We have $K_t(\ell(x_1), \ell(x_2)) = k_t(a(x_1 - x_2))$.  Therefore, if we define $p_t(x) = k_t(a(x))$ and let $P_t$ be the operator on $\R$ with integral kernel $P_t(x,y) = p_t(x-y)$, we have $\langle b \phi, A_t b \phi \rangle = \langle b \phi, P_t b \phi \rangle$.  Define

\bes
I_\beta = [ - t - \beta/2, - t + \beta/2] \cup [ t - \beta/2, t + \beta/2],
\ees
and write $b \phi = \phi_1 + \phi_2$, where the Fourier transform of $\phi_2$ is supported on $I_\beta$ and the transform of $\phi_1$ is supported on $\R \setminus I_\beta$.  Because $b$ was a fixed smooth function, we have $\| \phi_2 \|_2 \ll_A \beta^{-A} \ll_{\epsilon, A} t^{-A}$.  Because the kernel of $P_t$ is translation invariant, we have

\begin{align*}
\langle b \phi, P_t b \phi \rangle & = \langle \phi_1, P_t \phi_1 \rangle + \langle \phi_2, P_t \phi_2 \rangle \\
& \le \underset{\lambda \notin I_\beta}{\sup} |\widehat{p_t}(\lambda)| + O_{\epsilon,A}(t^{-A}) \| \widehat{p_t} \|_\infty.
\end{align*}
By Lemma 2.6 of \cite{Ma} (see also Lemma 4.1 of \cite{BGT}) we have

\be
\label{ptbound}
p_t(x) \ll t(1 + tx)^{-1/2},
\ee
and this imples that $\| \widehat{p_t} \|_\infty \ll t^{1/2}$.  It therefore suffices to prove the following estimate.

\begin{lemma}

We have $| \widehat{p_t}(\lambda) | \ll_\epsilon t^{1/2 + \epsilon} \beta^{-1/2}$ for $\lambda \notin I_\beta$.

\end{lemma}

\begin{proof}

Let $b_1\in C^\infty_0(\R)$ be a cutoff function that is equal to 1 on $[ -1, 1]$ and zero outside $[-2,2]$.  We wish to estimate the integral

\bes
\int_{-\infty}^\infty p_t(x) e^{i \lambda x} dx = \int_{-\infty}^\infty b_1(x) k_t(a(x)) e^{i \lambda x} dx
\ees
for $\lambda \notin I_\beta$.  Inverting the Harish-Chandra transform gives

\bes
\int_{-\infty}^\infty b_1(x) k_t(a(x)) e^{i \lambda x} dx = \frac{1}{2\pi} \int_{-\infty}^\infty \int_0^\infty b_1(x) \varphi_s(a(x)) e^{i \lambda x} h_t^2(s) s \tanh(\pi s) ds dx,
\ees
see for instance \cite{Se}.  If $s \in [0,\infty) \setminus [t-\beta/4, t+\beta/4]$, our assumption that $\beta \ge 2t^\epsilon$ implies that $(1+|s|) h_t(s) \ll_{\epsilon,A} t^{-A}$.  As $s \tanh(\pi s) \ll 1 + |s|$, this gives

\bes
\int_{-\infty}^\infty b_1(x) k_t(a(x)) e^{i \lambda x} dx = \frac{1}{2\pi} \int_{-\infty}^\infty \int_{t-\beta/4}^{t+\beta/4} b_1(x) \varphi_s(a(x)) e^{i \lambda x} h_t^2(s) s \tanh(\pi s) ds dx + O(t^{-A}).
\ees

It therefore suffices to prove the bound

\bes
\int_{-\infty}^\infty b_1(x) \varphi_s(a(x)) e^{i \lambda x} dx \ll_\epsilon t^{-1/2 + \epsilon} \beta^{-1/2}
\ees
uniformly for $\lambda \notin I_\beta$ and $s \in [t-\beta/4, t+\beta/4]$.  We decompose the integral as

\bes
\int_{-\infty}^\infty b_1( s^{-\epsilon} \beta x) \varphi_s(a(x)) e^{i \lambda x} dx + \int_{-\infty}^\infty (b_1(x) - b_1( s^{-\epsilon} \beta x) ) \varphi_s(a(x)) e^{i \lambda x} dx.
\ees
Our assumption that $\beta \ge 2 t^\epsilon$ implies that $s^{-\epsilon} \beta \ge 1$ for $t$ suficiently large.  Theorem 1.3 of \cite{Ma} gives the bound $\varphi_s(a(x)) \ll (1 + sx)^{-1/2}$ for $x \in [-2,2]$, and this implies that the first integral is $\ll_\epsilon t^{-1/2 + \epsilon} \beta^{-1/2}$.  To bound the second integral, by combining Proposition 4.12 of \cite{Ma} with either Lemma \ref{ann2} below or Proposition 4.13 of \cite{Ma} and applying stationary phase, we may prove that

\be
\label{phiasymp}
\varphi_s(a(x)) = c_1(x) e^{isx} (sx)^{-1/2} + c_2(x) e^{-isx} (sx)^{-1/2} + O( (sx)^{-3/2}),
\ee
where $c_i \in C^\infty(\R)$ and the error term is uniform for $x \in [-2,2] \setminus \{ 0 \}$.  As we have

\bes
\int_{s^{-1}}^1 (xs)^{-3/2} dx \ll s^{-1} \ll t^{-1/2 + \epsilon} \beta^{-1/2},
\ees
we may ignore the contribution to the second integral coming from the error term in (\ref{phiasymp}).  The two main terms in the asymptotic are identical, and so we shall treat the second one by estimating the integral

\bes
\int_{-\infty}^\infty (b_1(x) - b_1( s^{-\epsilon} \beta x) ) e^{i (\lambda-s) x} c_2(x) (sx)^{-1/2} dx.
\ees
After changing variable from $x$ to $s^{-\epsilon} \beta x$, this becomes

\bes
s^{-1/2+\epsilon} \beta^{-1/2} \int_{-\infty}^\infty (b_1(s^{-\epsilon} \beta x) - b_1(x) ) e^{i (\lambda-s) s^\epsilon \beta^{-1} x} c_2(s^\epsilon \beta^{-1} x) x^{-1/2} dx.
\ees
As $s^{-\epsilon} \beta \ge 1$, all derivatives of $b_1(s^{-\epsilon} \beta x) - b_1(x)$ and $c_2(s^\epsilon \beta^{-1} x)$ are bounded.  Moreover, all derivatives of $x^{-1/2}$ are bounded on the support of $b_1(s^{-\epsilon} \beta x) - b_1(x)$.  As $|\lambda-s| s^\epsilon \beta^{-1} \gg t^\epsilon$, repeated integration by parts implies that this integral is $\ll_{\epsilon,A} t^{-A}$ as required.

\end{proof}

\section{Spectral estimation of Hecke returns}
\label{sect4}

We now prove Theorem \ref{conditional} by improving the amplifier used in Proposition \ref{L2onspec}.  Our new ingredient is a spectral method for estimating the number of times the Hecke operators map $\ell$ close to itself, which allows us to prove the following result.

\begin{prop}
\label{thickbd}

Assume that $\psi$ satisfies (\ref{thick}) and (\ref{raman}).  We have the bound

\bes
\| \Pi_\beta b \psi \|_2 \ll_\epsilon t^{\theta/2 + \epsilon} \beta^{1/4 - \theta/2}.
\ees

\end{prop}

Theorem \ref{conditional} follows by choosing $\beta = t^{(1-2\theta)/(2-2\theta)}$ and combining this with Proposition \ref{L2offspec}.  We maintain the notations of Section \ref{ampperiod}.  Let $\epsilon > 0$ be given, and let $N$ be an integer of size roughly $t^{1/2 + \epsilon} \beta^{-1/2}$.  Define $\cT_1$ to be the operator

\bes
\cT_1 = \sum_{N/2 < p < N} \frac{\lambda(p)}{\sqrt{p}} T_p.
\ees
It again suffices to bound the inner product $\langle b \phi, \cT_1 \cT_1^* A_t b \phi \rangle$.  After reducing $\cT_1 \cT_1^*$ using the Hecke relations, we have

\be
\label{geopair2}
\langle b \phi, \cT_1 \cT_1^* A_t b \phi \rangle = \sum_{N/2 < p < N} I(t, \ell, \ell) + \sum_{N/2 < p_1,p_2 < N} \lambda(p_1) \overline{\lambda(p_2)} \frac{1}{\sqrt{p_1p_2}} \sum_{ \gamma \in R(p_1p_2)} I(t, \phi, \ell, \gamma \ell).
\ee

The key difference between the proof of Proposition \ref{L2onspec} and Proposition \ref{thickbd} is that we shall now estimate the recurrences of $\ell$ under a large collection of Hecke operators $T_n$ at once using spectral methods, rather than individually.  This is carried out in the following proposition.

\begin{prop}
\label{Hecke2}

If $M$ and $1 > \delta > 0$ satisfy $M \ge \delta^{-2-\epsilon}$, we have 

\bes
\sum_{ \substack{ M/2 < m < M \\ (m,q) = 1 } } \frac{1}{\sqrt{m}} M(\ell, m, \delta) \ll_\epsilon \delta^2 M^{3/2},
\ees
where $q$ is the integer defined in Section \ref{notation}.

\end{prop}

\begin{proof}

Let $b \in C^\infty_0(\g)$ be a real non-negative function that is supported in the ball of radius 2 about the origin with respect to the norm $\| \cdot \|$ defined in (\ref{gnorm}), and equal to 1 on the ball of radius 1.  Let $C_1 > 0$ be a constant to be chosen later.  Define $b_\delta \in C^\infty_0(\g)$ by $b_\delta(X) = b(\delta^{-1} C_1 X)$, and let $\widetilde{b}_\delta \in C^\infty_0(PSL_2(\R))$ be the pushforward of $b_\delta$ under $\exp$.

Let $\tilde{\ell} \subset PSL_2(\R)$ be the set obtained by extending $\ell$ by three times its length in both directions and lifting to $PSL_2(\R)$.  Let $\delta_{\widetilde{\ell}}$ be the length measure on $\widetilde{\ell}$, and let $f = \delta^{-2} \widetilde{b}_\delta * \delta_{\widetilde{\ell}}$.  If we choose $C_1$ to be small enough, the conditions $d(\ell, g\ell) \le 1$ and $n(\ell, g\ell) \le \delta$ imply that $\langle f, g f \rangle \gg 1$, where the implied constant is independent of $\delta$ and $\ell$.  If we define $\overline{f} \in L^2(\Gamma \backslash PSL_2(\R))$ by

\bes
\overline{f}(g) = \sum_{\gamma \in \Gamma} f(\gamma g),
\ees
then $\| \overline{f} \|_2 \sim 1$ in $L^2(\Gamma \backslash PSL_2(\R))$.\\

Choose $g \in C^\infty_0(0, \infty)$ to be real, positive, and satisfy $g(x) = 1$ for $1/2 \le x \le 1$.  If we define

\bes
\cS = \sum_{(m,q) = 1} \frac{g(m/M)}{\sqrt{m}} T_m,
\ees
then we have

\bes
\sum_{ \substack{ M/2 < m < M \\ (m,q) = 1 } } \frac{1}{\sqrt{m}} M(\ell, m, \delta) \ll \sum_{(m,q) = 1} \frac{g(m/M)}{\sqrt{m}} \sum_{\gamma \in R(m)} \langle f, \gamma f \rangle = \langle \overline{f}, \cS \overline{f} \rangle
\ees
and we may estimate the RHS spectrally.  Expand $\overline{f}$ with respect to a decomposition of $L^2(\Gamma \backslash PSL_2(\R))$ into automorphic representations as

\bes
\overline{f} = \sum_i \alpha_i \psi_i,
\ees
where $\psi_i$ is an $L^2$ normalised vector in an automorphic representation with eigenvalue $\mu_i$ under the Casimir operator $C$.  We have

\bes
\| C^n \overline{f} \|_2 \ll_n \delta^{- 2n}.
\ees
Integration by parts then gives

\begin{align*}
\langle \overline{f}, \psi_i \rangle & = \mu_i^{-n} \langle \overline{f}, C^n \psi_i \rangle \\
& = \mu_i^{-n} \langle C^n \overline{f}, \psi_i \rangle \\
& \ll_n |\mu_i|^{-n} \delta^{-2n},
\end{align*}
which implies that

\bes
\overline{f} = \langle \overline{f}, 1 \rangle + \sideset{}{^\prime}\sum_{|\mu_i| \le \delta^{-2-\epsilon/2} } \alpha_i \psi_i + O_{A, \epsilon}(\delta^A).
\ees
Note that we have normalised the volume of $\Gamma \backslash PSL_2(\R)$ to be 1, and $\Sigma'$ denotes the sum over the nontrivial representations.  Substituting this into $\langle \overline{f}, \cS \overline{f} \rangle$ gives

\bes
\langle \overline{f}, \cS \overline{f} \rangle = \langle \overline{f}, 1 \rangle^2 \sum_{(m,q) = 1} g(m/M) \sqrt{m} + \sum_{|\mu_i| \le \delta^{-2-\epsilon} } |\alpha_i|^2 \sum_{(m,q) = 1} g(m/M) \lambda_i(m) + O_{A, \epsilon}(M^{3/2}\delta^A),
\ees
where $\lambda_i(m)$ are the Hecke eigenvalues of $\psi_i$.  The result now follows from Lemma \ref{Lfunction} below, and the asymptotic $\langle \overline{f}, 1 \rangle \ll \delta$.  (Note that our assumptions that $M \ge \delta^{-2-\epsilon}$ and $|\mu_i| \le \delta^{-2 - \epsilon/2}$ guarantee that the hypothesis of the Lemma is satisfied.)

\end{proof}

\begin{lemma}
\label{Lfunction}
 
If $M \ge |\mu_i|^{1+\epsilon}$, we have

\bes
\sum_{(m,q) = 1} g(m/M) \lambda_i(m) \ll_{A,\epsilon} M^{-A},
\ees
where the implied constant is uniform in $\psi_i$.

\end{lemma}

\begin{proof}
 
We shall drop the subscript $i$, and assume that $\psi$ is a vector in a principal series representation as the discrete series case is similar.  We first consider the case $q = 1$.

Let $r$ be the spectral parameter of $\psi$, so that $\mu = 1/4 + r^2$.  By applying the functional equation and Stirling's formula, we see that the $L$-function $L(s, \psi)$ satisfies the estimate

\be
\label{convex}
L( -A + it, \psi) \ll_{A, \epsilon} ( t^2 + r^2 + 1)^{A + 1/2 + \epsilon}
\ee
for $A$ sufficiently large.  If we let $\widehat{g}(s)$ be the Mellin transform of $g$, which is entire and decays rapidly in vertical strips, we obtain

\bes
\sum_m g(m/M) \lambda(m) = \int_{(2)} L(s, \psi) \widehat{g}(s) M^s ds.
\ees
If we shift the line of integration to $\sigma = -A$, and apply (\ref{convex}) and the rapid decay of $\widehat{g}$, we have

\begin{align*}
\sum_m g(m/M) \lambda_i(m) & \ll_{A,\epsilon'} M^{-A} ( 1 + r^2)^{A + 1/2 + \epsilon'} \\
& \ll_{A,\epsilon'} M^{-A} \mu^{A + 1/2 + \epsilon'} \\
& \ll_{A,\epsilon'} M^{-A} M^{(1 - \epsilon)(A + 1/2 + \epsilon')} \\
& \ll_{B, \epsilon} M^{-B}
\end{align*}
as required.  In the case when $q > 1$, we apply the same argument to the incomplete $L$-function obtained by removing the local factors at primes dividing $q$ from $L(s, \psi)$.

\end{proof}

With these results, we are ready to estimate the RHS of (\ref{geopair2}).  We begin by applying the trivial bound of Proposition \ref{Ibound2} to the first sum, and our assumption that $|\lambda(p)| \le 2p^\theta$ to the second, which gives

\bes
\langle b \phi, \cT_1 \cT_1^* A_t b \phi \rangle \ll Nt^{1/2} + N^{2\theta} \sum_{N/2 < p_1,p_2 < N} \frac{1}{\sqrt{p_1p_2}} \sum_{ \gamma \in R(p_1p_2)} |I(t, \ell, \gamma \ell)|.
\ees
Enlarging the sum to one over all $N^2/4 < n < N^2$ with $(n,q) = 1$ gives

\be
\label{enlarge}
\langle b \phi, \cT_1 \cT_1^* A_t b \phi \rangle \ll Nt^{1/2} + N^{2\theta} \sum_{ \substack{ n \sim N^2 \\ (n,q) = 1} } \frac{1}{\sqrt{n}} \sum_{ \gamma \in R(n)} |I(t, \ell, \gamma \ell)|.
\ee
By Proposition \ref{Ibound2}, we only need to consider the terms in the second sum with $d(\ell, g\ell) \le 1$ and $n(\ell, g\ell) \le t^{-1/2+\epsilon} \beta^{1/2}$, which gives

\bes
\sum_{ \substack{ n \sim N^2 \\ (n,q) = 1} } \frac{1}{\sqrt{n}} \sum_{ \gamma \in R(n)} |I(t, \ell, \gamma \ell)| \ll \sum_{ \substack{ n \sim N^2 \\ (n,q) = 1} } \frac{t^{1/2}}{\sqrt{n}} M(\ell, n, t^{-1/2+\epsilon} \beta^{1/2}).
\ees
The assumption that $N \sim t^{1/2 + \epsilon} \beta^{-1/2}$ implies that we may choose $\delta = t^{-1/2+\epsilon} \beta^{1/2}$ and $M = N^2$ in Proposition \ref{Hecke2}, so that

\bes
\sum_{ \substack{ n \sim N^2 \\ (n,q) = 1} } \frac{t^{1/2}}{\sqrt{n}} M(\ell, n, t^{-1/2+\epsilon} \beta^{1/2}) \ll N^3 t^{-1 + \epsilon} \beta.
\ees
Substituting this into (\ref{enlarge}) gives

\bes
| \langle \cT_1 A_t^0 \psi, b \phi \rangle |^2 \le \langle b \phi, \cT_1 \cT_1^* A_t b \phi \rangle \ll Nt^{1/2} + N^{3 + 2\theta} t^{-1/2+\epsilon} \beta.
\ees
If we estimate the action of $\cT_1$ on $\psi$ using our assumption (\ref{thick}) and substitute $N \sim t^{1/2 + \epsilon} \beta^{-1/2}$, we obtain

\bes
| \langle \psi, b \phi \rangle | \ll_\epsilon t^{\theta/2 + \epsilon} \beta^{1/4 - \theta/2}
\ees
as required.

\remark The method we have used of estimating Hecke recurrences spectrally is unlikely to work in other situations.  It requires us to choose an amplifier that makes the sums of eigenvalues in Proposition \ref{Hecke2} longer than the relevant analytic conductors, and in other cases (such as higher rank or when using the operators $T_{p^2}$ on $GL_2$ to give an unconditional theorem) this gives the amplifier so much mass that the `off-diagonal' term is worse than the trivial bound.  The method also depends on the exponent of $\kappa$ in Proposition \ref{Ibound2} being small, and fails to improve the $L^\infty$ bound of \cite{IS} under the assumption (\ref{thick}) because the corresponding exponent in that case is larger.

\section{Oscillatory Integrals When $\lambda \sim t$}
\label{sect5}

In this section, we prove Proposition \ref{Ibound2} by building up the integral $I(t, \phi, \ell_1, \ell_2)$ in several steps.  We begin with two calculations that we shall use repeatedly in this section and in Section \ref{sect6}.

\begin{lemma}
\label{Kdiff}

Fix $g \in PSL_2(\R)$, and define $\sigma : \R / 2\pi \Z \rightarrow \R / 2\pi \Z$ by $k(\theta) g \in NA k(\sigma(\theta))$.  Then $\sigma$ is a diffeomorphism.

\end{lemma}

\begin{proof}

By using the Cartan decomposition, we may reduce to the case where $g = a(y)$.  Taking inverses gives $a(-y) k(-\theta) \in k(-\sigma(\theta)) AN$, and applying both sides to the point at infinity gives

\be
\label{upangle}
e^{-y} \cot(\theta/2) = \cot(\sigma(\theta)/2).
\ee
This proves that $\sigma$ is a bijection, and a diffeomorphism everywhere except at $\theta = 0$.  Rewriting the equation as $e^y \tan (\theta/2) = \tan( \sigma(\theta)/2)$ proves it at $\theta = 0$ also.

\end{proof}

\begin{lemma}
\label{Adiff}

Let $g \in PSL_2(\R)$ have Iwasawa decomposition $g  = n a k(\theta)$.  Then

\bes
\frac{\partial}{\partial t} A( g a(t)) \Big|_{t=0} = \cos \theta.
\ees

\end{lemma}

\begin{proof}

If $H$ is as in (\ref{Hdef}), we have

\begin{align*}
A( g a(t)) & = A(a) + A( k(\theta) \exp(tH) k(-\theta) ) \\
& = A(a) + A( \exp( t \Ad(k(\theta)) H) ),
\end{align*}
and therefore

\begin{align*}
\frac{\partial}{\partial t} A( g a(t)) \Big|_{t=0} & = H^*( \Ad(k(\theta)) H ) \\
& = \cos \theta.
\end{align*}

\end{proof}

\subsection{Uniformisation results}

We shall need the following two uniformisation lemmas for the function $A$.

\begin{lemma}
\label{ann1}

Let $D > 0$.  There exists $\delta > 0$, $\sigma > 0$, and a real analytic function $\xi : (-\delta, \delta) \times (-D,D)^2 \rightarrow \R$ such that

\bes
\frac{\partial}{\partial y} A( k(\theta) n(x) a(y) ) = 1 - \theta^2 \xi(\theta, x, y)
\ees
and

\begin{align}
\label{xilower}
|\xi(\theta, x, y)| & \ge \sigma \\
\notag
\left| \frac{\partial^n \xi}{\partial y^n} (\theta, x, y) \right| & \ll_n 1
\end{align}
for $(\theta,x,y) \in (-\delta, \delta) \times (-D,D)^2$.

\end{lemma}

\begin{proof}

Define the function $\alpha(\theta, x, y) : \R / 2\pi \Z \times (-2D, 2D)^2 \rightarrow \R / 2\pi \Z$ by requiring that

\bes
k(\theta) n(x) a(y) \in NA k(\alpha(\theta, x, y)).
\ees
The analyticity of the Iwasawa decomposition implies that $\alpha$ is analytic as a function of $(\theta, x, y)$.  Lemma \ref{Adiff} implies that

\begin{align*}
1 - \frac{\partial}{\partial y} A( k(\theta) n(x) a(y) ) & = 1 - \cos \alpha \\
& = 2 \sin^2(\alpha/2).
\end{align*}
We choose $\delta$ such that $\sin(\alpha/2)$ vanishes on $(-2\delta, 2\delta) \times (-2D,2D)^2$ iff $\theta = 0$.  Lemma \ref{Kdiff} implies that $\partial \alpha / \partial \theta$ never vanishes on $\{ 0 \} \times (-2D, 2D)^2$, and so because $\alpha$ was analytic we see that there is a real analytic function $\xi_0$ on $(-2\delta, 2\delta) \times (-2D,2D)^2$ such that $\sin(\alpha/2) = \theta \xi_0$.  Defining $\xi = 2 \xi_0^2$ and restricting the domain to $(-\delta, \delta) \times (-D, D)^2$ gives the result.

\end{proof}

\begin{lemma}
\label{ann2}

If $I \subset \R$ is a bounded open interval, there exists $\delta > 0$ and a function $\xi : I \times (-\delta, \delta) \rightarrow \R$ such that $\xi(y,0) = 0$ for all $y \in I$,

\be
\label{Aunif}
A(k(\theta) a(y)) = y - y \xi^2(y,\theta),
\ee
and the map $\Xi : (y, \theta) \mapsto (y, \xi(y,\theta))$ gives a real-analytic diffeomorphism $\Xi : I \times (-\delta, \delta) \simeq U \subset \R^2$.

\end{lemma}

\begin{proof}

This follows in the same way as Lemma \ref{ann1} above, or Theorem 4.6 of \cite{Ma}.

\end{proof}

\subsection{Constituent integrals of $I(t, \phi, \ell_1, \ell_2)$}

We now estimate two one-dimensional integrals that appear in $I(t, \phi, \ell_1, \ell_2)$.

\begin{prop}
\label{equalline1}

Let $C$, $D$ and $\epsilon$ be positive constants, and let $b \in C^\infty_0(\R)$ be a function supported in $[0,1]$.  If $x, y \in [-D, D]$ and

\be
\label{lineassump1}
|\theta| \ge C s^{-1/2 + \epsilon} \beta^{1/2} \quad \text{and} \quad |\lambda - s| \le \beta
\ee
for some $s$ and $\beta$ satisfying $1 \le \beta \ll s^{2/3}$, then

\be
\label{lineint1}
\int_{-\infty}^\infty b(z) \exp( i\lambda z - is A(k(\theta) n(x) a(y + z))) dz \ll_{A} s^{-A}
\ee
uniformly in $\lambda$ and $\beta$.

\end{prop}

\begin{proof}

By applying Lemma \ref{ann1}, we see that there is some $\delta > 0$ and a nonvanishing real analytic function $\xi$ on $(-\delta, \delta) \times (-D-2, D+2)^2$ such that

\bes
\frac{\partial}{\partial z} A( k(\theta) n(x) a(y+z) ) = 1 - \theta^2 \xi(\theta, x, y+z)
\ees
when $\theta \in (-\delta, \delta)$, $x,y \in [-D, D]$ and $z \in [0,1]$.  If $Z(\theta, x, y)$ is an antiderivative of $\xi$ with respect to $y$, we may integrate this to obtain

\bes
A( k(\theta) n(x) a(y) ) = y - \theta^2 Z(\theta, x, y) + c(x,\theta).
\ees

If $\theta \in (-\delta, \delta)$, we may use this to rewrite the integral (\ref{lineint1}) as

\begin{align}
\notag
\int_{-\infty}^\infty b(z) \exp( i\lambda z - is A( k(\theta) n(x) a(y+z) ) ) dy & = e^{i c(\theta, x) - isy} \int_{-\infty}^\infty b(z) \exp( i(\lambda - s)z + is \theta^2 Z(\theta, x, y+z)) dz \\
\label{plane2}
& = e^{i c(\theta, x) - isy} \int_{-\infty}^\infty b(z) \exp( is \theta^2 \Psi(z) ) dz,
\end{align}
where we define $\Psi(z) = Z(\theta, x, y+z) + s^{-1} \theta^{-2}(s-\lambda) z$.

Our assumption (\ref{lineassump1}) implies that

\bes
| s^{-1} \theta^{-2}(s-\lambda) | \le s^{-1} \theta^{-2} \beta \ll s^{-2\epsilon},
\ees
so that

\be
\label{partialpsi}
\Psi = Z(\theta, x, y+z) + O(s^{-2\epsilon})z, \quad \text{and} \quad \frac{\partial \Psi}{\partial z} = \xi(\theta, x, y+z) + O(s^{-2\epsilon}).
\ee
It follows from (\ref{xilower}) and (\ref{partialpsi}) that for $s$ sufficiently large, $|\partial \Psi / \partial z| > \sigma/2$ for all $\theta \in (-\delta, \delta)$, $x,y \in [-D, D]$ and $z \in [0,1]$.  As (\ref{lineassump1}) implies that $s \theta^2 \gg s^{2\epsilon} \beta \ge s^{2\epsilon}$, the bound (\ref{lineint1}) follows by integration by parts in (\ref{plane2}).\\

In the case where $\theta \notin (-\delta, \delta)$, Lemma \ref{Adiff} implies that $(\partial / \partial z) A( k(\theta) n(x) a(y+z) ) \le 1 - c_1$ for some $c_1 > 0$ depending only on $\delta$, which gives

\bes
\frac{\partial}{\partial z} ( i\lambda s^{-1} z - A( k(\theta) n(x) a(y+z) ) \gg 1.
\ees
The result now follows by integration by parts.\\

\end{proof}

The second one-dimensional integral that we shall estimate is as follows.

\begin{prop}
\label{equalline2}

Let $C$, $D$ and $\epsilon$ be positive constants, and let $b \in C^\infty_0(\R)$ be a function supported in $[0,1]$.  If $x, y \in [-D, D]$ and

\be
\label{lineassump2}
|x| \ge C s^{-1/2 + \epsilon} \beta^{1/2} \quad \text{and} \quad |\lambda - s| \le \beta
\ee
for some $s$ and $\beta$ satisfying $1 \le \beta \ll s^{2/3}$, then

\be
\label{lineint2}
\int_{-\infty}^\infty b(z) e^{i\lambda z} \varphi_{-s}(n(x) a(y+z)) dz \ll_{A} s^{-A}
\ee
uniformly in $\lambda$ and $\beta$.

\end{prop}

\begin{proof}

If we substitute the formula for $\varphi_{-s}$ as an integral of plane waves into the LHS of (\ref{lineint2}), it becomes

\bes
\int_{-\infty}^\infty \int_0^{2\pi} b(z) \exp( i\lambda z + (1/2 - is) A( k(\theta) n(x) a(y+z) ) )  d\theta dz.
\ees
Let $f(x) \in C^\infty_0(\R)$ be a function with $\text{supp}(f) \subseteq [-2,2]$ and $f(x) = 1$ on $[-1,1]$.  Let $C_1$ be a positive constant to be chosen later.  Define $b_1$ by $b_1(x) = f( C_1^{-1} s^{1/2 -\epsilon} \beta^{-1/2} x)$ and set $b_2 = 1 - b_1$, so that $1 = b_1(\theta) + b_2(\theta)$ is a smooth partition of unity on $\R / 2\pi \Z$ with

\begin{align*}
\text{supp}(b_1) & \subseteq [-2C_1 s^{-1/2 +\epsilon} \beta^{1/2}, 2C_1 s^{-1/2 +\epsilon} \beta^{1/2}], \\
\text{supp}(b_2) & \subseteq \R / 2\pi \Z \setminus [-C_1 s^{-1/2 +\epsilon} \beta^{1/2}, C_1 s^{-1/2 +\epsilon} \beta^{1/2}].
\end{align*}
Proposition \ref{equalline1} implies that

\bes
\int_{-\infty}^\infty \int_0^{2\pi} b_2(\theta) b(z) \exp( i\lambda z + (1/2 - is) A( k(\theta) n(x) a(y+z) ) )  d\theta dz \ll_A s^{-A},
\ees
so that it suffices to estimate

\bes
\int_{-\infty}^\infty \int_0^{2\pi} b_1(\theta) b(z) \exp( i\lambda z + (1/2 - is) A( k(\theta) n(x) a(y+z) ) )  d\theta dz.
\ees
We shall do this by estimating the integrals

\be
\label{equallinetheta}
\int_0^{2\pi} b_1(\theta) \exp( - is A( k(\theta) n(x) a(y) ) )  d\theta
\ee
in $\theta$, where now $y \in [-D, D+1]$.  

If $X \in \g$, we let $X^*$ be the vector field on $\bH$ whose value at $p$ is $\tfrac{\partial}{\partial t} \exp(tX) p |_{t=0}$.  It may be shown that these vector fields satisfy $[X^*, Y^*] = -[X,Y]^*$, where the first Lie bracket is on $\bH$ and the second is in $\g$.  We recall the vectors $X_\gn$ and $X_\gk$ defined in (\ref{Hdef}).  It may be easily seen that the subset of $\bH$ where $X_\gk^* A$ vanishes is exactly $A$, and the following lemma implies that it vanishes to first order there.

\begin{lemma}
\label{Avanish}

We have $X_\gn^* X_\gk^* A(a(y)) = e^y$ for all $y$.

\end{lemma}

\begin{proof}

We have

\bes
X_\gn^* X_\gk^* A = X_\gk^* X_\gn^* A + [X_\gn^*, X_\gk^*] A.
\ees
It may be seen that the first term vanishes, and we have

\bes
[X_\gn^*, X_\gk^*] = - [X_\gn, X_\gk]^* = H^*
\ees
which implies the lemma.

\end{proof}

Lemma \ref{Avanish} implies that there exist $\sigma$, $\delta > 0$ such that if $|x| < \sigma$ and $y \in [-D-1, D+2]$ then we have $|X_\gk A(n(x)a(y))| \ge \delta |x|$.  Define

\begin{align*}
B & = \{ n(x) a(y) | \; |x| \le \sigma /2, \, y \in [-D, D+1] \}, \\
B' & = \{ n(x) a(y) | \; |x| < \sigma, \, y \in [-D-1, D+2] \}.
\end{align*}
Let $p \in B$ and assume that $|N(p)| \ge C s^{-1/2+\epsilon} \beta^{1/2}$, where $N(p)$ is as in (\ref{Iwasawa}).  If $s$ is sufficiently large and $C_1$ sufficiently small, and $|\theta| \le 2 C_1 s^{-1/2 +\epsilon} \beta^{1/2}$, we have $k(\theta)p \in B'$ and $|N(k(\theta)p)| \gg s^{-1/2+\epsilon} \beta^{1/2}$.  It follows that

\bes
\Big| \frac{\partial}{\partial \theta'} A( k(\theta') p) \Big|_{\theta' = \theta} \Big| = | X_\gk^* A( k(\theta) p) | \ge \delta |N( k(\theta)p)| \gg s^{-1/2+\epsilon} \beta^{1/2}
\ees
when $|\theta| \le 2 C_1 s^{-1/2 +\epsilon} \beta^{1/2}$.  The proposition now follows by integration by parts.

If $p = n(x) a(y)$ with $x \in [-D, D]$ and $y \in [-D, D+1]$ and $p \notin B$, then we have $|X_\gk^* A(k(\theta)p) | \ge \delta$ when $|\theta| \le 2 C_1 s^{-1/2 +\epsilon} \beta^{1/2}$ and $s$ is sufficiently large.  The proposition again follows by integration by parts.

\end{proof}

\subsection{Proof of Proposition \ref{Ibound2}}

We now combine Propositions \ref{equalline1} and \ref{equalline2} to bound the integral (\ref{equalpairint}) below, which will imply Proposition \ref{Ibound2} after integrating in the various spectral parameters.

\begin{prop}
\label{equalpair1}

Let $\ell \subset \bH$ be a unit geodesic segment with parametrisation $\ell : [0,1] \rightarrow \bH$.  Let $D \subset PSL_2(\R)$ be a compact set, let $b_1, b_2 \in C^\infty_0(\R)$ be functions supported in $[0,1]$, and let $\epsilon > 0$ be given.  If $g \in D$ and $\lambda_1$, $\lambda_2 \in \R$ satisfy

\bes
n(\ell, g\ell) \ge s^{-1/2 + \epsilon} \beta^{1/2} \quad \text{and} \quad \lambda_i \in [s-\beta, s+\beta]
\ees
for some $s$ and $1 \le \beta \ll s^{2/3}$, then

\be
\label{equalpairint}
\iint_{-\infty}^\infty b_1(x_1) b_2(x_2) e^{i(\lambda_1 x_1 - \lambda_2 x_2)} \varphi_{-s}(\ell(x_1), g\ell(x_2)) dx_1 dx_2 \ll_{A, \epsilon} s^{-A}
\ee
uniformly in $\lambda_i$ and $\beta$.

\end{prop}

\begin{proof}

We begin by expressing $\varphi_{-s}$ as an integral of plane waves.  For $y, z \in \bH$ we have

\bes
\varphi_s(y,z) = \int_0^{2\pi} \exp( (1/2 - is)( A(k_z(\sigma)y) - A( k_z(\sigma)z)) d\sigma,
\ees
where $K_z$ is the stabilizer of $z$ and $k_z : \R / 2\pi \Z \rightarrow K_z$ is a parametrisation.  Define the function $\theta : \R / 2\pi \Z \rightarrow \R / 2\pi \Z$ by

\bes
k_z(\sigma) \in NA k(\theta(\sigma)).
\ees
Lemma \ref{Kdiff} implies that $\theta$ is a diffeomorphism.  Because $A(k_z(\sigma)y) - A( k_z(\sigma)z) = A( k(\theta(\sigma)) y) - A( k(\theta(\sigma)) z)$, we have

\be
\label{phirep}
\varphi_s(y,z) = \int \exp( (1/2 - is)( A(k(\theta)y) - A( k(\theta)z)) \frac{d\sigma}{d\theta} d\theta.
\ee
We may assume that $\ell$ is the segment with one endpoint at $i$ and pointing upwards, so that $\ell(x) = a(x) i$.  Substituting (\ref{phirep}) into (\ref{equalpairint}) gives

\begin{multline}
\label{tripleint1}
\iint_{-\infty}^\infty b_1(x_1) b_2(x_2) e^{i(\lambda_1 x_1 - \lambda_2 x_2)} \varphi_{-s}(\ell(x_1), g\ell(x_2)) dx_1 dx_2 = \\
\iint_{-\infty}^\infty \int_0^{2\pi} b_1(x_1) b_2(x_2) e^{i(\lambda_1 x_1 - \lambda_2 x_2)} \exp( (1/2 - is)( A(k(\theta) a(x_1)) - A( k(\theta)g a(x_2))) \frac{d\sigma}{d\theta} d\theta dx_1 dx_2.
\end{multline}
Let $g = k(\theta')n(x') a(y')$, where $x'$ and $y'$ are bounded in terms of $D$.  We then have $k(\theta)g a(x_2) = k(\theta + \theta') n(x') a(y + y')$.  We integrate the RHS of (\ref{tripleint1}) with respect to $x_1$ and $x_2$ with $\theta$ fixed.  Choose a constant $C > 0$.  If $\theta \notin [-C s^{-1/2+\epsilon} \beta^{1/2}, C s^{-1/2+\epsilon} \beta^{1/2}]$, then Proposition \ref{equalline1} implies that the integral is $\ll s^{-A}$, and likewise if $\theta + \theta' \notin [-C s^{-1/2+\epsilon} \beta^{1/2}, C s^{-1/2+\epsilon} \beta^{1/2}]$.  Combining these, we see that (\ref{tripleint1}) will be $\ll s^{-A}$ unless $|\theta'| \le 2C s^{-1/2+\epsilon} \beta^{1/2}$.

If $C$ is chosen sufficiently small, the condition $|\theta'| \le 2C s^{-1/2+\epsilon} \beta^{1/2}$ and our assumption that $n(\ell, g\ell) \ge s^{-1/2 + \epsilon} \beta^{1/2}$ imply that $\ell$ and $g \ell$ are separated in the sense that there is a $C_1 > 0$ such that

\bes
d(p, g\ell) \ge C_1 s^{-1/2 + \epsilon} \beta^{1/2}
\ees
for all $p \in \ell$.  The result now follows by applying Proposition \ref{equalline2} to the integral of the LHS of (\ref{tripleint1}) over $x_2$ for each fixed $x_1$.

\end{proof}

\begin{cor}
\label{equalpair2}

Let $\ell \subset \bH$ be a unit geodesic segment with parametrisation $\ell : [0,1] \rightarrow \bH$.  Let $D \subset PSL_2(\R)$ be a compact set, and let $b_1, b_2 \in C^\infty_0(\R)$ be functions supported in $[0,1]$.  Let $\epsilon > 0$, $s > 0$, and $1 \le \beta \ll s^{2/3}$ be given.  Let $\phi \in L^2(\R)$ be a function with $\| \phi \|_2 = 1$ and such that

\bes
\textup{supp}(\widehat{\phi}) \subseteq [s-\beta, s+\beta].
\ees
If $g \in D$ and $n(\ell, g\ell) \ge s^{-1/2 + \epsilon} \beta^{1/2}$, then

\bes
\iint_{-\infty}^\infty b_1(x_1) b_2(x_2) \phi(x_1) \overline{\phi(x_2)} \varphi_{-s}(\ell(x_1), g\ell(x_2)) dx_1 dx_2 \ll_{A, \epsilon} s^{-A},
\ees
where the implied constant is independent of $\phi$ and $\beta$.

\end{cor}

\begin{proof}

This follows immediately from Proposition \ref{equalpair1} after inverting the Fourier transform of $\phi$ and noting that $\| \widehat{\phi} \|_1 \le \| \widehat{\phi} \|_2 (2\beta)^{1/2} = (2\pi)^{1/2} (2\beta)^{1/2}$. 

\end{proof}

\begin{proof}[Proof of Proposition \ref{Ibound2}]

To prove the bound (\ref{Ibound21}), observe that equation (\ref{ptbound}) implies that

\bes
\int_{-\infty}^\infty |K_t(\ell_1(x), p)|^2 dx \ll t
\ees
uniformly for $p \in \bH$.  It follows that $K_t(\ell_1(x_1), \ell_2(x_2))$ has norm $\ll t^{1/2}$ as an element of $L^2(\R^2)$, and the result follows by Cauchy-Schwarz.

We now prove (\ref{Ibound22}).  Fix a unit geodesic segment $\ell$.  We may assume without loss of generality that $\ell_1 = \ell$, and we choose $g \in PSL_2(\R)$ so that $g \ell = \ell_2$.  The assumption that $d(\ell_1, \ell_2) \le 1$ implies that $g$ lies in the compact set $D := \{ g \in PSL_2(\R) | d(\ell, g \ell) \le 1 \}$.  We have

\bes
I(t, \phi, \ell, g\ell) = \iint_{-\infty}^\infty b(x_1) b(x_2) \phi(x_1) \overline{\phi(x_2)} K_t(\ell(x_1), g\ell(x_2)) dx_1 dx_2.
\ees
Inverting the Harish-Chandra transform of $k_t$ gives

\bes
I(t, \phi, \ell, g\ell) = \frac{1}{2\pi} \iint_{-\infty}^\infty \int_0^\infty b(x_1) b(x_2) \phi(x_1) \overline{\phi(x_2)} h_t^2(s) \varphi_{-s}(\ell(x_1), g \ell(x_2)) s \tanh(\pi s) ds dx_1 dx_2.
\ees
If we assume without loss of generality that $\beta > t^\epsilon$, then we may restrict the domain of the Harish-Chandra transform to $[t-\beta, t+\beta]$ as in Section \ref{boundsoffspec} to obtain

\begin{multline*}
I(t, \phi, \ell, g\ell) = \frac{1}{2\pi} \iint_{-\infty}^\infty \int_{t-\beta}^{t+\beta} b(x_1) b(x_2) \phi(x_1) \overline{\phi(x_2)} \\
h_t^2(s) \varphi_{-s}(\ell(x_1), g \ell(x_2)) s \tanh(\pi s) ds dx_1 dx_2 + O(t^{-A}).
\end{multline*}
Applying Corollary \ref{equalpair2} with $2\beta$ in place of $\beta$ completes the proof.

\end{proof}

\section{Oscillatory Integrals When $\lambda < t$}
\label{sect6}

We now prove Proposition \ref{Ibound1}.  In this section, we assume that all geodesics we consider carry an orientation.  When we refer to the unit tangent vector to a geodesic at a point, we shall always mean in the direction of its orientation.  If $\ell_1$ and $\ell_2$ are two intersecting geodesics, we shall denote by $\angle(\ell_1, \ell_2)$ the angle between their unit tangent vectors at the point of intersection measured in the counterclockwise direction from $\ell_1$ to $\ell_2$.

Let $\ell$ be the vertical geodesic through $i$.  By slight abuse of notation, we take $a : \R \rightarrow \ell$ to be a parametrisation of $\ell$, and define $\ell_0 = a([0,1])$ which is a unit segment contained in $\ell$.  We give the geodesic $\ell$ the upwards-pointing orientation, which we transfer to $g \ell$ for $g \in PSL_2(\R)$.  As in the proof of Proposition \ref{Ibound2}, it suffices to bound the integral

\bes
I(s, \lambda, g) = \iint_{-\infty}^\infty e^{i \lambda(x_1-x_2)} b_1(x_1)b_2(x_2) \varphi_{-s}( \ell(x_1), g\ell(x_2)) dx_1 dx_2.
\ees
After substituting the expression (\ref{phirep}) for $\varphi_{-s}( \ell(x_1), g\ell(x_2))$, we obtain an oscillatory integral in the variables $\theta$, $x_1$, and $x_2$ with phase function

\bes
\phi(x_1, x_2, \theta, g, \rho) = \rho(x_1 - x_2) - A( k(\theta) \ell(x_1)) + A( k(\theta) g \ell(x_2)),
\ees
where $\rho = \lambda / s \ge 0$.  We first assume that $\rho \in [\delta, 1-\delta]$ for some $1/2 > \delta > 0$.  Define $\alpha \in [0,\pi/2]$ to be the solution to $\cos \alpha = \rho$, which is bounded away from $0$ and $\pi/2$.  We shall study the critical points of $\phi$ in Sections \ref{osc1} to \ref{osc4}, before deriving a bound for $I(s, \lambda, g)$ from our results in Section \ref{osc5}.  We shall write $\phi(x_1, x_2, \theta)$ when $g$ and $\rho$ are not varying.

\subsection{The critical points of $\phi$}
\label{osc1}

\begin{lemma}
\label{critical}

The phase function $\phi$ has a critical point at $(x_1, x_2, \theta, g, \rho)$ exactly when $k(\theta) \ell(x_1)$ and $k(\theta) g \ell(x_2)$ lie on the same vertical geodesic $v$, which we give the upwards-pointing orientation, and we have $\angle(v, k(\theta) \ell), \angle(v, k(\theta) g \ell) \in \{ \pm \alpha \}$.

\end{lemma}

\begin{proof}

Suppose that $(x_1', x_2', \theta')$ is a critical point of $\phi$.  Define the functions $x(\theta)$, $y(\theta)$ and $\beta(\theta)$ by

\bes
k(\theta) a(x_1') = n(x(\theta)) a( y(\theta)) k(\beta(\theta)),
\ees
and let $n' = n(x(\theta'))$ and $\beta' = \beta(\theta')$.  It may be seen that $v := n(x') \ell$ is the upwards-pointing geodesic through $k(\theta') \ell(x_1')$, and that $\beta' = \angle(v, k(\theta') \ell)$.  Lemma \ref{Adiff} then implies that $\beta' = \pm \alpha$.  The calculation in the case of $\partial / \partial x_2$ is identical.

We have

\begin{align*}
A( k(\theta) a(x_1')) - A( k(\theta) g a(x_2')) & = A( k(\theta) a(x_1')) - A( k(\theta) a(x_1') a(x_1')^{-1} g a(x_2')) \\
& = y(\theta) - A( a(y(\theta)) k(\beta(\theta)) a(x_1')^{-1} g a(x_2')) \\
& = - A( k(\beta(\theta)) a(x_1')^{-1} g a(x_2'))
\end{align*}
and so

\begin{align}
\notag
\frac{\partial \phi}{\partial \theta} (x_1', x_2', \theta') & = \frac{\partial}{\partial \theta} A( k(\beta(\theta)) a(x_1')^{-1} g a(x_2')) \Big|_{\theta = \theta'} \\
\label{phithetadiff}
& = \frac{\partial \beta}{\partial \theta} (\theta')  \frac{\partial}{\partial \beta} A( k(\beta) a(x_1')^{-1} g a(x_2')) \Big|_{\beta = \beta'}.
\end{align}
Because $\partial \beta/\partial \theta$ does not vanish by Lemma \ref{Kdiff}, and

\bes
\frac{\partial}{\partial \theta} A( k(\theta) g ) \Big|_{\theta=0} = 0
\ees
iff $g \in AK$, we have $\partial \phi / \partial \theta = 0$ iff $k(\beta') a(x_1')^{-1} g a(x_2') \in AK$, i.e. $k(\beta') a(x_1')^{-1} g a(x_2')i$ lies on the vertical geodesic through the origin.  Because $k(\beta') a(x_1)^{-1} = a'^{-1} n'^{-1} k(\theta')$, this is equivalent to the condition that $k(\theta') g a(x_2') \in n' AK$, or that $k(\theta') g \ell(x_2')$ lies on the vertical geodesic $v$ passing through $k(\theta') \ell(x_1')$.

We finish with an observation that will be useful in calculating the Hessian of $\phi$.  We have $k(\beta') a(x_1')^{-1} g a(x_2') \in a(h)K$ for some $h \in \R$, and it may be seen that $k(\theta') \ell(x_1') \in n' a' K$ and $k(\theta') g \ell(x_2') \in n' a' a(h) K$, so that $h$ is the signed distance from $k(\theta') \ell(x_1')$ to $k(\theta') g \ell(x_2')$ along $v$.

\end{proof}

Given a pair of geodesics $\ell_1$ and $\ell_2$, we say that a geodesic $j$ is a critical geodesic for $(\ell_1, \ell_2)$ if $j$ meets $\ell_1$ and $\ell_2$ at angles of $\pm \alpha$.  We may therefore rephrase Lemma \ref{critical} as saying that $(x_1, x_2, \theta, g, \rho)$ is a critical point of $\phi$ exactly when $(\ell, g\ell)$ has a critical geodesic $j$, $\ell(x_1)$ and $g \ell(x_2)$ both lie on $j$, and $k(\theta)j$ is vertical.  As in Lemma \ref{critical}, we define the aperture of a critical point to be the signed distance from $\ell(x_1)$ to $g \ell(x_2)$ on the geodesic $j$.

We shall now calculate the Hessian of $\phi$ at its critical points.  Let $(x_1', x_2', \theta')$ be a critical point of $\phi$, and define functions $\beta_i(\theta)$ by

\be
\label{Iwasawapair}
k(\theta) a(x_1') \in NA k(\beta_1(\theta)), \quad k(\theta) g a(x_2') \in NA k(\beta_2(\theta)).
\ee
We let $\beta_i' = \beta_i(\theta')$.  It follows from Lemma \ref{critical} that $\beta_i' \in \{ \pm \alpha \}$.  Let $h$ be the aperture of the critical point, so that

\bes
k(\beta_1') a(x_1')^{-1} g a(x_2') \in \left( \begin{array}{cc} e^{h} & 0 \\ 0 & 1 \end{array} \right) K.
\ees
We define $\kappa = \tfrac{\partial \beta_1}{\partial \theta} (\theta')$, which is nonzero by Lemma \ref{Kdiff}.  The Hessian of $\phi$ at $(x_1', x_2', \theta')$ is given by the following proposition.

\begin{prop}
\label{Hessian}

The Hessian of $\phi$ at $(x_1', x_2', \theta')$ with respect to the co-ordinates $(x_1, x_2, \theta)$ is

\bes
D = \left( \begin{array}{ccc} \tfrac{1}{2} \sin^2 \alpha & 0 & \kappa \sin \beta_1' \\ 0 & -\tfrac{1}{2} \sin^2 \alpha & -\kappa e^h \sin \beta_2' \\  \kappa \sin \beta_1' & -\kappa e^h \sin \beta_2' & \kappa^2 (1 - e^{2h})/2 \end{array} \right)
\ees
The determinant of $D$ is

\bes
|D| = \frac{3}{8} \kappa^2 \sin^4 \alpha (1 - e^{2h}),
\ees
which is nonzero unless $h = 0$, i.e. the points $\ell(x_1')$ and $g\ell(x_2')$ coincide in $\bH$.

\end{prop}

\begin{proof}

It is clear that $\partial^2 \phi / \partial x_1 \partial x_2$ is identically 0.  To calculate $\partial^2 \phi / \partial x_1^2$, define $\gamma : \R \rightarrow \R / 2\pi \Z$ by the condition that $k(\theta') a(x_1' + t) \in NA k (\gamma(t))$.  Our assumption that we are at a critical point implies that $\gamma(0) = \beta_1' = \pm \alpha$.  Lemma \ref{Adiff} gives

\bes
\frac{\partial}{\partial t} \phi(x_1' + t, x_2', \theta') = \rho - \cos \gamma(t),
\ees
and

\be
\label{xxdiff}
\frac{\partial^2 \phi}{\partial x_1^2} (x_1', x_2', \theta') = \sin \beta_1' \frac{\partial \gamma}{\partial t} (0).
\ee
We have

\begin{align*}
k(\theta') a(x_1' + t) & \in NA k (\gamma(t)) \\
NA k( \beta_1') a(t) & = NA k (\gamma(t)) \\
k( \beta_1') a(t) & \in NA k (\gamma(t)).
\end{align*}
Equation (\ref{upangle}) then gives $\tan (\gamma(t)/2) = e^t \tan( \beta_1'/2)$, so that

\begin{align*}
\frac{\partial \gamma}{\partial t} \sec^2 (\gamma(t)/2) & = e^t \tan( \beta_1'/2) \\
\frac{\partial \gamma}{\partial t} (0) & = \cos^2 (\beta_1'/2) \tan( \beta_1'/2) \\
 & = \frac{1}{2} \sin \beta_1'.
\end{align*}
Substituting this into (\ref{xxdiff}) gives

\bes
\frac{\partial^2 \phi}{\partial x_1^2} (x_1', x_2', \theta') = \frac{1}{2} \sin^2 \beta_1' = \frac{1}{2} \sin^2 \alpha.
\ees
The calculation of $\partial^2 \phi / \partial x_2^2$ is identical, with the exception of a change in sign.\\

To calculate $\partial^2 \phi / \partial \theta \partial x_1$, we again have

\bes
\frac{\partial \phi}{\partial x_1} (x_1', x_2', \theta) = \rho - \cos \beta_1(\theta),
\ees
and

\bes
\frac{\partial^2 \phi}{\partial \theta \partial x_1} (x_1', x_2', \theta') = \sin \beta_1' \frac{\partial \beta_1}{\partial \theta} (\theta') = \kappa \sin \beta_1'.
\ees
We likewise have

\bes
\frac{\partial^2 \phi}{\partial \theta \partial x_2} \phi(x_1', x_2', \theta') = -\sin  \beta_2' \frac{\partial \beta_2}{\partial \theta} (\theta'),
\ees
and we shall express $\tfrac{\partial \beta_2}{\partial \theta} (\theta')$ in terms of $\kappa$ and $h$.  We recall that

\bes
k( \beta_1') a(x_1')^{-1} g a(x_2') = a(h) k(\theta_0)
\ees
for some $\theta_0$, and so

\bes
k(\theta) a( x_1') k( -\beta_1') a(h) k(\theta_0) = k(\theta) g a(x_2').
\ees
Substituting both parts of (\ref{Iwasawapair}) into this gives

\begin{align*}
NA k( \beta_1(\theta)) k( -\beta_1') a(h) k(\theta_0) & = NA k( \beta_2(\theta)) \\
k( \beta_1(\theta) - \beta_1')a(h) & \in NA k( \beta_2(\theta) -\theta_0).
\end{align*}
By setting $\theta = \theta'$ we see that $\theta_0 = \beta_1'$.  Equation (\ref{upangle}) then gives

\bes
e^h \tan( (\beta_1(\theta) - \beta_1')/2) = \tan( (\beta_2(\theta) - \beta_2')/2),
\ees
and differentiating both sides with respect to $\theta$ and evaluating at $\theta = \theta'$ gives

\bes
\frac{\partial \beta_2}{\partial \theta} (\theta') = \kappa e^h.
\ees
It follows that

\bes
\frac{\partial^2 \phi}{\partial \theta \partial x_2} \phi(x_1', x_2', \theta') = - \kappa e^h \sin \beta_2'.
\ees
\\

To calculate $\partial^2 \phi / \partial \theta^2$, we have as in (\ref{phithetadiff}) that

\bes
\frac{\partial \phi}{\partial \theta}(x_1', x_2', \theta) = \frac{\partial \beta_1}{\partial \theta} \frac{\partial}{\partial \beta} A( k(\beta) a(x_1')^{-1} g a(x_2')) \Big|_{\beta = \beta_1(\theta)}.
\ees
Because

\bes
\frac{\partial}{\partial \beta} A( k(\beta) a(x_1')^{-1} g a(x_2')) \Big|_{\beta = \beta_1'} = 0,
\ees
we have

\begin{align*}
\frac{\partial^2 \phi}{\partial \theta^2} (x_1', x_2', \theta') & = \kappa^2 \frac{\partial^2}{\partial \beta^2} A( k(\beta) a(x_1')^{-1} g a(x_2')) \Big|_{\beta = \beta_1'} \\
& = \kappa^2 \frac{\partial^2}{\partial \beta^2} A( k(\beta - \beta_1') a(h) ) \Big|_{\beta = \beta_1'}.
\end{align*}
It is a standard calculation (see for instance Proposition 4.4 of \cite{Ma}) that

\bes
\frac{\partial^2}{\partial \beta^2} A( k(\beta) a(h) ) \Big|_{\beta = 0} = (1 - e^{2h})/2,
\ees
and this completes the proof.

\end{proof}

\subsection{The function $\psi$}
\label{osc2}

Define $\cP = \R / 2\pi \Z \times PSL_2(\R) \times [\delta, 1-\delta]$, and define $\cS \subset \cP$ to be the set where one of the geodesics $k(\theta) \ell$ and $k(\theta) g \ell$ is vertical.  Note that $\cS$ is closed, and contains at most $4$ values of $\theta$ for each fixed $(g, \rho)$.  We may define functions

\bes
\xi_1, \xi_2 : \cP \setminus \cS \rightarrow \R
\ees
by requiring that $k(\theta) a(\xi_1(\theta, g, \rho))$ is the unique point on $k(\theta)\ell$ at which the tangent vector to the geodesic makes an angle of $\alpha$ with the upward pointing vector, and likewise for $\xi_2(\theta, g, \rho)$ and $k(\theta) g \ell$.  As $\xi_1$ does not depend on $g$, we will omit this argument of the function.  We have 

\be
\label{redangles}
k(\theta) a(\xi_1(\theta, \rho)) \in NA k(\epsilon_1 \alpha), \quad k(\theta) g a(\xi_2(\theta, g, \rho)) \in NA k(\epsilon_2 \alpha)
\ee
for $\epsilon_i \in \{ \pm 1 \}$, and so equation (\ref{upangle}) gives

\bes
e^{\xi_1(\theta, \rho)} \tan(\theta/2) = \tan( \epsilon_1 \alpha/2), \quad e^{\xi_2(\theta, g, \rho)} \tan(\theta/2) = \tan( \epsilon_2 \alpha/2).
\ees
Moreover, it may be seen that $\epsilon_1 = 1$ iff the geodesic $k(\theta) \ell$ runs from right to left in the upper half plane model of $\bH$, which is equivalent to $\theta \in (0, \pi)$, and likewise for $\epsilon_2$.

It follows from Lemma \ref{Adiff} that $\xi_1(\theta, \rho)$ and $\xi_2(\theta, g, \rho)$ may also be characterised as the unique functions such that

\be
\label{redcondition}
\frac{\partial \phi}{\partial x_1}(\xi_1(\theta, \rho), x_2, \theta) = \frac{\partial \phi}{\partial x_2}(x_1, \xi_2(\theta, g, \rho), \theta) = 0.
\ee
We define

\begin{align*}
\psi : \cP \setminus \cS & \rightarrow \R \\
\psi(\theta, g, \rho) & = \phi( \xi_1(\theta, \rho), \xi_2(\theta, g, \rho), \theta, g, \rho).
\end{align*}

\begin{lemma}
\label{redHessian}

$(\theta', g', \rho')$ is a critical point of $\psi$ exactly when $(\xi_1(\theta'), \xi_2(\theta'), \theta', g', \rho')$ is a critical point of $\phi$.  If $(\theta', g', \rho')$ is a critical point of $\psi$, let $\kappa$ and $h$ be the values associated to the corresponding critical point of $\phi$.  We then have

\bes
\frac{\partial^2 \psi}{\partial \theta^2} (\theta', g', \rho') = -\frac{3}{2} \kappa^2 (1 - e^{2h}).
\ees

\end{lemma}

\begin{proof}

We shall fix $g$ and $\rho$, and omit them from the arguments of $\phi$ and $\psi$.  Let $D$ be the Hessian of $\phi$ calculated in Proposition \ref{Hessian}.  If we apply the chain rule to $\psi$ and substitute $\theta = \theta'$, we obtain

\bes
\frac{\partial^2 \psi}{\partial \theta^2}(\theta') = ( \tfrac{ \partial \xi_1}{\partial \theta}(\theta'), \tfrac{ \partial \xi_2}{\partial \theta}(\theta'), 1) D ( \tfrac{ \partial \xi_1}{\partial \theta}(\theta'), \tfrac{ \partial \xi_2}{\partial \theta}(\theta'), 1)^t.
\ees
To calculate $\tfrac{ \partial \xi_1}{\partial \theta}(\theta')$ and $\tfrac{ \partial \xi_2}{\partial \theta}(\theta')$, we differentiate (\ref{redcondition}) with respect to $\theta$ and set $\theta = \theta'$ to obtain

\be
\label{Hesssub}
\frac{\partial^2 \phi}{\partial \theta \partial x_1} (\xi_1(\theta'), \xi_2(\theta'), \theta') + \frac{ \partial \xi_1}{\partial \theta}(\theta') \frac{\partial^2 \phi}{\partial x_1^2} (\xi_1(\theta'), \xi_2(\theta'), \theta') = 0.
\ee
Substituting the second partial derivatives of $\phi$ calculated in Proposition \ref{Hessian} gives

\bes
\frac{ \partial \xi_1}{\partial \theta}(\theta') = \frac{-2 \kappa}{\sin \beta_1'},
\ees
and likewise

\bes
\frac{ \partial \xi_2}{\partial \theta}(\theta') = \frac{-2 \kappa e^h}{\sin \beta_2'}.
\ees
The lemma follows on substituting these into equation (\ref{Hesssub}).

\end{proof}

It follows that the set of $(g,\rho)$ for which the function $\psi(\theta, g, \rho)$ has a degenerate critical point are exactly those for which either $\ell = g \ell$ or $\angle(\ell, g\ell) = \pm 2\alpha$.  Note that these two cases are distinct, as $\alpha \in (0, \pi/2)$.  In the first case the function $\psi(\theta, g, \rho)$ vanishes identically.  In the second case, $\psi(\theta, g, \rho)$ has only a single degenerate critical point, as no oriented geodesic can cross $\ell$ and $g \ell$ making an angle of $\alpha$ with both except at their point of intersection.  To determine this critical point, the condition that $\angle(\ell, g\ell) = \pm 2\alpha$ implies that $g \in a(y) k(\pm 2\alpha) A$ for some $y \in \R$, so that $\ell \cap g \ell = a(y)i$.  The angle bisector of the two geodesics at the point $a(y)i$ is $a(y) k(\pm \alpha) \ell$, and the critical point of $\psi(\theta, g, \rho)$ is the $\theta$ such that the positive endpoint of $k(\theta) a(y) k(\pm \alpha) \ell$ is $i \infty$.  This is equivalent to the condition  $k(\theta) a(y) k(\pm \alpha) \in NA$, and equation (\ref{upangle}) then gives $\cot \theta/2 = \mp e^y \cot \alpha/2$.

We define

\begin{align*}
\cD_1 & = \{ (\theta, g, \rho) \in \cP \setminus \cS | g \in A \} \\
\cD_2^\pm & = \{ (\theta, g, \rho) \in \cP \setminus \cS | g \in a(y) k(\pm 2\alpha) A , \cot \theta/2 = \mp e^y \cot \alpha/2 \}
\end{align*}
to be the three sets on which $\psi$ has a degenerate critical point.  We also define $\overline{\cP} = PSL_2(\R) \times [\delta, 1-\delta]$, and define

\begin{align*}
\overline{\cD}_1 & = A \times [\delta, 1-\delta] \\
\overline{\cD}_2^\pm & = \{ (g, \rho) \in \overline{\cP} | g \in A k(\pm 2\alpha) A \}
\end{align*}
to be the projections of $\cD_1$ and $\cD_2^\pm$ to $\overline{\cP}$.

\subsection{The degenerate set $\cD_1$}
\label{osc3}

As $\psi(\theta, g, \rho) = \psi(\theta, ga, \rho)$ for $a \in A$, we may study the degeneracy of $\psi$ near $\cD_1$ by differentiating $\psi(\theta, \exp(X), \rho)$ at $X = 0$ as in the following proposition.

\begin{prop}
\label{psidiff}

If $X = \left( \begin{array}{cc} 0 & X_1 \\ X_2 & 0 \end{array} \right) \in \g$, then

\be
\label{psiliediff}
\frac{\partial}{\partial t} \psi(\theta, \exp(tX), \rho) \Big|_{t=0} = \epsilon \sin \alpha ( e^{-\xi_2(\theta, e, \rho)} X_1 + e^{\xi_2(\theta, e, \rho)} X_2 ),
\ee
where $\epsilon$ is $1$ if $\theta \in (0,\pi)$ and $-1$ otherwise.  In particular, $\partial \psi / \partial t (\theta, \exp(tX), \rho) |_{t=0}$ has no degenerate critical points unless $X = 0$.

\end{prop}

\begin{proof}

Let $x_1' = \xi_1(\theta, \rho)$ and $x_2' = x_2(\theta, e, \rho)$.  We have

\begin{align*}
\frac{\partial}{\partial t} \psi(\theta, \exp(tX), \rho) \Big|_{t=0} & = \frac{\partial}{\partial t} \phi( \xi_1(\theta, \rho), \xi_2(\theta, \exp(tX), \rho), \theta, \exp(tX), \rho) \Big|_{t=0} \\
& = \frac{\partial \phi}{\partial x_2} ( \xi_1(\theta, \rho), \xi_2(\theta, e, \rho), \theta, e, \rho) \frac{\partial}{\partial t} \xi_2(\theta, \exp(tX), \rho) \Big|_{t=0} \\
& \qquad + \frac{\partial}{\partial t} \phi( \xi_1(\theta, \rho), \xi_2(\theta, e, \rho), \theta, \exp(tX), \rho) \Big|_{t=0}.
\end{align*}
The first term vanishes by (\ref{redcondition}), so we are left with 

\begin{align*}
\frac{\partial}{\partial t} \psi(\theta, \exp(tX), \rho) \Big|_{t=0} & = \frac{\partial}{\partial t} \phi( \xi_1(\theta, \rho), \xi_2(\theta, e, \rho), \theta, \exp(tX), \rho) \Big|_{t=0} \\
& = \frac{\partial}{\partial t} A( k(\theta) \exp(tX) a(\xi_2(\theta, e, \rho)) ) \Big|_{t=0}.
\end{align*}

We shall abbreviate $\xi_2(\theta, e, \rho)$ to $\xi_2(\theta)$ for the remainder of the proof.  Write the first order approximation to the Iwasawa decomposition of $k(\theta) \exp(tX) a(\xi_2(\theta))$ as

\bes
k(\theta) \exp(tX) a(\xi_2(\theta)) = n \exp( t X_N + O(t^2)) a \exp( t X_A + O(t^2)) k \exp( t X_K + O(t^2)),
\ees
where $X_N \in \gn$, $X_A \in \ga$, and $X_K \in \gk$.  As in equation (\ref{redangles}), we have $k = k(\alpha)$ if $\theta \in (0,\pi)$ and $k = k(-\alpha)$ if $\theta \in (-\pi, 0)$.  We first assume that $\theta \in (0, \pi)$.  Rearranging and equating first order terms gives

\begin{align*}
X & = \Ad(a(\xi_2(\theta)) k(\alpha)^{-1} a^{-1}) X_N + \Ad(a(\xi_2(\theta)) k(\alpha)^{-1}) X_A + \Ad(a(\xi_2(\theta))) X_K \\
\Ad( k(\alpha) a(\xi_2(\theta))^{-1}) X & = \Ad(a^{-1}) X_N + X_A + \Ad(k(\alpha)) X_K
\end{align*}
As $\Ad(a^{-1}) X_N$ and $\Ad(k(\alpha)) X_K$ lie in $\ga^\perp \subset \g$, we see that $X_A$ is the orthogonal projection of $\Ad( k(\alpha) a(\xi_2(\theta))^{-1}) X$ to $\ga$.  A calculation gives

\bes
X_A = \sin \alpha (e^{-\xi_2(\theta)} X_1 + e^{\xi_2(\theta)} X_2 ) \left( \begin{array}{cc} 1/2 & 0 \\ 0 & -1/2 \end{array} \right),
\ees
so that

\bes
\frac{\partial}{\partial t} A( k(\theta) \exp(tX) a(\xi_2(\theta)) ) \Big|_{t=0} = \sin \alpha (e^{-\xi_2(\theta)} X_1 + e^{\xi_2(\theta)} X_2 ).
\ees
This proves (\ref{psiliediff}) when $\theta \in (0,\pi)$, and the other case is identical.

We now prove that $\partial \psi / \partial t (\theta, \exp(tX), \rho) |_{t=0}$ has no degenerate critical points if $X \neq 0$ and $\theta \in (0, \pi)$.  We define $f(x) = \sin \alpha (X_1 e^{-x} + X_2 e^x)$, so that

\bes
\frac{\partial}{\partial t} A( k(\theta) \exp(tX) a(\xi_2(\theta)) ) \Big|_{t=0} = f( \xi_2(\theta)).
\ees
Differentiating equation (\ref{upangle}) gives

\bes
\frac{ \partial \xi_2}{\partial \theta} = -\frac{1}{2} e^{-\xi_2(\theta)} \tan (\alpha/2) \csc^2(\theta/2),
\ees
so that $\partial \xi_2 / \partial \theta$ is always nonzero.  Suppose that $X \neq 0$, and that $\theta$ is a degenerate critical point of $\partial \psi / \partial t (\theta, \exp(tX), \rho) |_{t=0}$.  We then have

\begin{align*}
0 & = \frac{\partial^2}{\partial \theta \partial t} A( k(\theta) \exp(tX) a(\xi_2(\theta)) ) \Big|_{t=0} \\
& = f'( \xi_2(\theta)) \frac{ \partial \xi_2}{\partial \theta} \\
& = f'( \xi_2(\theta)).
\end{align*}
Differentiating again with respect to $\theta$ gives

\begin{align*}
0 & = \frac{\partial^3}{\partial^2 \theta \partial t} A( k(\theta) \exp(tX) a(\xi_2(\theta)) ) \Big|_{t=0} \\
& = f''( \xi_2(\theta)) \left( \frac{ \partial \xi_2}{\partial \theta} \right)^2 \\
& = f''( \xi_2(\theta)),
\end{align*}
but this is a contradiction as it may be easily checked that $f$ has no degenerate critical points unless $X = 0$.  The case of $\theta \in (-\pi, 0)$ is identical.

\end{proof}

Define $P = \R / 2\pi \Z \times \ga^\perp \times [\delta, 1-\delta]$, and define $S = \{ (\theta, X, \rho) \in P | (\theta, \exp(X), \rho) \in \cS \}$.  $S$ is again closed, and contains at most 4 values of $\theta$ for each fixed $(X, \rho)$.

\begin{lemma}
\label{perturb1}

There is an open neighbourhood $0 \in U \subset \ga^\perp$ such that for all $X \in U$ and all $b \in C^\infty_0( P \setminus S)$ we have

\bes
\int b(\theta, X, \rho) e^{is \psi(\theta, \exp(X), \rho)} d\theta \ll (1 + s \|X\|)^{-1/2},
\ees
where $\| X \|$ is as in (\ref{gnorm}).

\end{lemma}

\begin{proof}

Define the map $X : \R \times \R / 2\pi \Z \rightarrow \ga^\perp$ by

\bes
X(r, \gamma) = \left( \begin{array}{cc} 0 & r \sin \gamma \\ r \cos \gamma & 0 \end{array} \right).
\ees
We define

\bes
\widetilde{P} = \R / 2\pi \Z \times \R \times \R / 2\pi \Z \times [\delta, 1-\delta]
\ees
and

\bes
\widetilde{S} = \{ (\theta, r, \gamma, \rho) \subset \widetilde{P} | (\theta, X(r,\gamma), \rho) \in S \}.
\ees
We define $\widetilde{b}(\theta, r, \gamma, \rho) \in C^\infty_0(\widetilde{P} \setminus \widetilde{S})$ and $\widetilde{\psi}(\theta, r, \gamma, \rho) \in C^\infty(\widetilde{P} \setminus \widetilde{S})$ to be the pullbacks of $b$ and $\psi$ under $X$.  We know that $\widetilde{\psi}$ vanishes when $r = 0$, and as $\widetilde{\psi}$ is smooth (in fact, real analytic) we have that $\widetilde{\psi}/r$ is again a smooth function.  Proposition \ref{psidiff} implies that $\widetilde{\psi}/r$ has no degenerate critical points when $r = 0$, and so there is some $\epsilon > 0$ such that it also has no degenerate critical points on the set $\text{supp}(\widetilde{b}) \cap (\R / 2\pi \Z \times [-\epsilon, \epsilon] \times \R / 2\pi \Z \times [\delta, 1-\delta])$.  If we define $U = X( (-\epsilon, \epsilon) \times \R / 2\pi \Z )$, the result now follows from stationary phase.

\end{proof}

\begin{cor}
\label{D1}

If $(a', \rho') \in \overline{\cD}_1$, there is an open neighbourhood $(a', \rho') \in U \subset \overline{\cP}$ such that for all $b \in C^\infty_0(\cP \setminus \cS)$ and all $(g, \rho) \in U$, we have

\bes
\int_0^{2\pi} b(\theta, g, \rho) e^{is \psi(\theta, g, \rho)} d\theta \ll (1 + s n(\ell_0, g \ell_0))^{-1/2}.
\ees

\end{cor}

\begin{proof}

Let $U_X \subset \ga^\perp$ be as in Lemma \ref{perturb1}.  If $g = \exp(X) a'$ for $X \in U_X$, we have $n(\ell_0, g \ell_0) \sim \| X \|$, where the implied constants depend on $a'$.  As $\psi(\theta, ga, \rho) = \psi(\theta, g, \rho)$ for $a \in A$, the result follows from Lemma \ref{perturb1}.

\end{proof}

\subsection{The degenerate set $\cD_2^\pm$}
\label{osc4}

The next proposition proves that $\psi$ has a cubic degeneracy on $\cD_2^\pm$.

\begin{prop}
\label{cubic}

If $(\theta', g', \rho') \in \cD_2^\pm$, we have $\partial^3 \psi / \partial \theta^3(\theta', g', \rho') \neq 0$.

\end{prop}

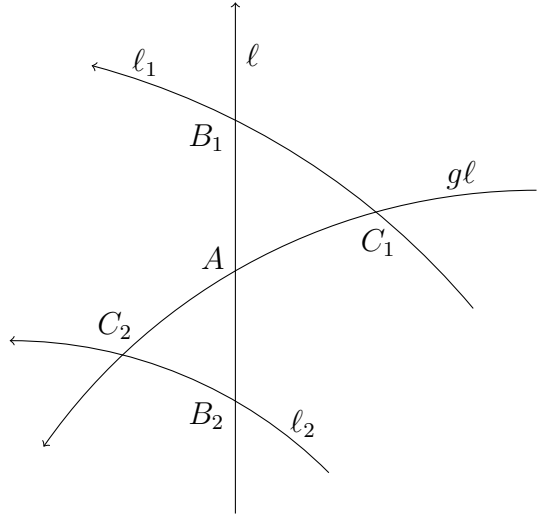
\begin{figure}
\begin{tikzpicture}
\draw[->] (4,3.7) -- (4,10.5);
\draw[<-] (1.45,4.59) arc (145:90:8);
\draw[<-] (1,6) arc (90:45:6);
\draw[<-] (2.09,9.66) arc (75:40:10);
\draw (4,9.8) node[anchor = west]{$\ell$};
\draw (7,8.2) node{$g\ell$};
\draw (2.8,9.7) node{$\ell_1$};
\draw (4.9,4.9) node{$\ell_2$};
\draw (4,7.1) node[anchor = east]{$A$};
\draw (4,8.7) node[anchor = east]{$B_1$};
\draw (4,5) node[anchor = east]{$B_2$};
\draw (5.9,7.3) node{$C_1$};
\draw (2.4,6.2) node{$C_2$};
\end{tikzpicture}
\caption{Two degenerating critical geodesics.}
\end{figure}

\begin{proof}

Suppose $g' = a(y) k(2\alpha) a_2$.  Define $g = a(y) k(2\alpha + \epsilon) a_2$ for some $\epsilon > 0$.  If $\epsilon$ is chosen sufficiently small, the pair $(\ell, g \ell)$ will have exactly two critical geodesics $\ell_1$ and $\ell_2$ as shown in Figure 1.  The triangles $AB_1C_1$ and $AB_2C_2$ both have angular defect, and hence area, $\epsilon$.  Our assumption that $\alpha$ was bounded away from $0$ and $\pi/2$ then implies that $AB_1 = AB_2 \sim \epsilon^{1/2}$ and $B_1C_1 = B_2C_2 \sim \epsilon^{1/2}$, where the implied constants depends only on $\delta$.  The critical points $\theta_i$ corresponding to $\ell_i$ are the solutions to

\bes
\cot \theta_1/2 = -e^{y+AB_1} \cot \alpha/2, \qquad \cot \theta_2/2 = -e^{y-AB_1} \cot \alpha/2.
\ees
It follows that $0 > \theta_1 > -\alpha > \theta_2 > -\pi$, and also that $\theta_1 - \theta_2 \sim \epsilon^{1/2}$.  The apertures $h_i$ of the critical points $\theta_i$ are given by $h_1 = - B_1C_1 \sim -\epsilon^{1/2}$ and $h_2 = B_2C_2 \sim \epsilon^{1/2}$, so that Lemma \ref{redHessian} gives

\bes
\frac{\partial^2 \psi}{\partial \theta^2} (\theta_1, g, \rho') \sim -\epsilon^{1/2}, \quad \frac{\partial^2 \psi}{\partial \theta^2} (\theta_2, g, \rho') \sim \epsilon^{1/2}.
\ees
It follows that there is $\theta_0 \in [\theta_2, \theta_1]$ at which

\bes
\frac{\partial^3 \psi}{\partial \theta^3} (\theta_0, g, \rho') = \frac{ \frac{\partial^2 \psi}{\partial \theta^2} (\theta_2, g, \rho') - \frac{\partial^2 \psi}{\partial \theta^2} (\theta_1, g, \rho') }{ \theta_2 - \theta_1} \sim -1,
\ees
and shrinking $\epsilon$ to 0 gives the result.  The case $g' \in A k(-2\alpha) A$ is identical.

\end{proof}

\begin{cor}
\label{D2}

If $(g', \rho') \in \overline{\cD}_2^\pm$, there is an open neighbourhood $(g', \rho') \in U \subset \overline{\cP}$ such that for all $b \in C^\infty_0(\cP \setminus \cS)$ and all $(g, \rho) \in U$, we have

\bes
\int_0^{2\pi} b(\theta, g, \rho) e^{is \psi(\theta, g, \rho)} d\theta \ll s^{-1/3}.
\ees

\end{cor}

\begin{proof}

By Proposition \ref{cubic}, there exists a neighbourhood $U_\theta$ of $\theta'$ and $U$ of $(g', \rho')$ such that $(U_\theta \times U) \cap \cS = \emptyset$, and $|\partial^3 \psi / \partial \theta^3| \ge \sigma > 0$ on $U_\theta \times U$.  As $\psi(\theta, g', \rho')$ only has a critical point at $\theta'$, by shrinking $U$ we may also assume that $\psi$ has no critical points on $(\R / 2\pi \Z \setminus U_\theta) \times U \setminus \cS$.  The result then follows from Proposition 2, Section 1.2, Chapter VIII of \cite{MS}.

\end{proof}

\subsection{Bounds for $I(t, \lambda, \ell_1, \ell_2)$}
\label{osc5}

We shall use the results of the previous sections to prove the follwing proposition, which implies Proposition \ref{Ibound1} in the case $\lambda / t \in [\delta, 1-\delta]$ after inverting the Harish-Chandra transform.

\begin{prop}
\label{Isgprop}

Let $D \subset PSL_2(\R)$ be a compact set, let $b_1, b_2 \in C^\infty_0(\R)$ be functions supported in $[0,1]$, and let $1/2 > \delta > 0$.  For $g \in PSL_2(\R)$ and $\lambda, s \in \R$, define

\be
\label{Isg}
I(s, \lambda, g) = \iint_{-\infty}^\infty e^{i \lambda(x_1-x_2)} b_1(x_1)b_2(x_2) \varphi_{-s}( \ell(x_1), g\ell(x_2)) dx_1 dx_2.
\ee
If $g \in D$ and $\lambda / s \in [\delta, 1-\delta]$, we have

\bes
I(s, \lambda, g) \ll \Big\{ \begin{array}{ll} s^{-1}(1 + s n(\ell_0, g\ell_0))^{-1/2} & \text{when} \quad n(\ell_0, g \ell_0) \le s^{-1/3} \\
s^{-4/3} & \text{when} \quad n(\ell_0, g \ell_0) \ge s^{-1/3}.
\end{array}
\ees

\end{prop}

\begin{proof}

If we substitute the expression (\ref{phirep}) into (\ref{Isg}), we obtain

\bes
\int_0^{2\pi} \iint_{-\infty}^\infty b_1(x_1) b_2(x_2) e^{i\lambda(x_1 - x_2)} \exp( (1/2 - is)( A(k(\theta) a(x_1)) - A( k(\theta)g a(x_2))) \frac{d\sigma}{d\theta} dx_1 dx_2 d\theta.
\ees
We let $b \in C^\infty_0(PSL_2(\R))$ be a function that is equal to 1 on $D$, and introduce a factor of $b(g)$ into the integral.  When $g \in D$ we then have

\bes
I(s, \lambda, g) = \int_0^{2\pi} \iint_{-\infty}^\infty e^{is \phi(x_1, x_2, \theta, g, \rho)} c(x_1, x_2, \theta, g, \rho) dx_1 dx_2 d\theta,
\ees
where $c \in C^\infty_0(\R^2 \times \cP)$ is the combination of all of the amplitude factors.  The following lemma eliminates the variables $x_1$ and $x_2$.

\begin{lemma}

There is a function $c_1 \in C^\infty_0( \cP \setminus \cS )$ such that for all $(g, \rho) \in D \times [\delta, 1-\delta]$ we have

\bes
I(s, \lambda, g) = s^{-1} \int_0^{2\pi} e^{is \psi(\theta, g, \rho)} c_1(\theta, g, \rho) d\theta + O(s^{-2}).
\ees

\end{lemma}

\begin{proof}

We shall apply stationary phase in the $x_i$ variables.  For fixed $(\theta, g, \rho)$, the function $\phi(x_1, x_2)$ has one critical point at $(\xi_1(\theta, \rho), \xi_2(\theta, g, \rho))$ if $(\theta, g, \rho) \notin \cS$, and none otherwise.  Moreover, it may be shown in the same way as the proof of Proposition \ref{Hessian} that the Hessian at this critical point is

\bes
D = \left( \begin{array}{cc} \tfrac{1}{2} \sin^2 \alpha & 0 \\ 0 & -\tfrac{1}{2} \sin^2 \alpha \end{array} \right),
\ees
so that the critical point is uniformly nondegenerate.

Define

\bes
\cP_0 = \{ (\theta, g, \rho) \in \cP \setminus \cS | (\xi_1(\theta, \rho), \xi_2(\theta, g, \rho), \theta, g, \rho) \in \text{supp}(c) \},
\ees
so that $\cP_0$ is compact and $\cP_0 \cap \cS = \emptyset$.  If we define $c_1 \in C^\infty_0(\cP \setminus \cS)$ by

\bes
c_1(\theta, g, \rho) = \frac{2\pi}{\sin^2 \alpha} c( \xi_1(\theta, \rho), \xi_2(\theta, g, \rho), \theta, g, \rho),
\ees
then we have $\text{supp}(c_1) \subseteq \cP_0$, and stationary phase gives

\be
\label{xireduce}
\iint_{-\infty}^\infty e^{is \phi(x_1, x_2, \theta)} c(x_1, x_2, \theta, g, \rho) dx_1 dx_2 = e^{is \psi(\theta, g, \rho)} s^{-1} c_1(\theta, g, \rho) + O(s^{-2})
\ee
locally uniformly on $\cP \setminus \cS$.  We also have

\bes
\iint_{-\infty}^\infty e^{is \phi(x_1, x_2, \theta)} c(x_1, x_2, \theta, g, \rho) dx_1 dx_2 \ll_A s^{-A}
\ees
locally uniformly on $\cP \setminus \cP_0$.  Therefore, if we extend $c_1$ to a function in $C^\infty(\cP)$ by making it 0 on $\cS$, then (\ref{xireduce}) holds locally uniformly on $\cP$ and the lemma follows.

\end{proof}

We now apply Corollaries \ref{D1} and \ref{D2}.  Corollary \ref{D1} implies that there is an open neighbourhood $U_1$ of $\overline{\cD}_1 \cap (D \times [\delta, 1-\delta])$ in $\overline{\cP}$ such that

\bes
I(s, \lambda, g) \ll s^{-1}(1 + s n(\ell_0, g\ell_0))^{-1/2}
\ees
when $(g, \rho) \in U_1 \cap (D \times [\delta, 1-\delta])$, and Corollary \ref{D2} implies that there is a neighbourhood $U_2$ of $\overline{\cD}_2^\pm \cap (D \times [\delta, 1-\delta])$ such that $I(s, \lambda, g) \ll s^{-4/3}$ when $(g, \rho) \in U_2 \cap (D \times [\delta, 1-\delta])$.  As $\psi$ has no degenerate critical points outside $\overline{\cD}_1 \cup \overline{\cD}_2^\pm$, we also have $I(s, \lambda, g) \ll s^{-3/2}$ when $(g, \rho) \in (D \times [\delta, 1-\delta]) \setminus (U_1 \cup U_2)$.  As the bound in Proposition \ref{Isgprop} is the maximum of these three bounds, this completes the proof.

\end{proof}

It remains to discuss the case when $\lambda = 0$, so that $\alpha = \pi/2$.  The proof proceeds as before, until the analysis of the degenerate critical points of $\psi$.  These degeneracies now occur when

\bes
g \in A \cup \left( \begin{array}{cc} 0 & 1 \\ -1 & 0 \end{array} \right) A,
\ees
and the function $\psi$ vanishes identically at these points.  These degeneracies may be treated in exactly the same way as $\cD_1$ in Section \ref{osc3}, which gives the bound

\bes
I(s, 0, g) \ll s^{-1}(1 + s n(\ell_0, g\ell_0))^{-1/2}.
\ees
Inverting the Harish-Chandra transform completes the proof.

\end{document}